\pdfoutput=1
\documentclass[11pt]{article}
\usepackage[margin=1in]{geometry}

\usepackage{amsmath, amssymb, amsfonts}
\usepackage{amsthm}
\usepackage{mathtools, dsfont, stackrel, booktabs, bbm}
\usepackage{multirow}
\usepackage{float}
\usepackage{capt-of}
\usepackage{url,verbatim}
\usepackage{enumitem}
\usepackage{xargs}
\usepackage{color}
\usepackage{hyperref}
\usepackage{microtype}
\usepackage{bm}

\usepackage{tikz}
\usetikzlibrary{
  matrix,
  shapes,
  arrows.meta,
  positioning,
  fit,
  backgrounds,
  patterns,
  calc,
  decorations.pathreplacing
}

\tikzset{
  block/.style={
    draw, rounded corners, thick,
    align=center,
    minimum height=1.1cm,
    minimum width=2.6cm,
    inner sep=6pt
  },
  smallblock/.style={
    draw, rounded corners, thick,
    align=center,
    minimum height=0.9cm,
    minimum width=2.2cm,
    inner sep=5pt
  },
  decision/.style={
    draw, diamond, thick,
    align=center,
    inner sep=2pt,
    minimum width=2.2cm,
    minimum height=1.2cm
  },
  arrow/.style={-Latex, thick},
  dashedbox/.style={draw, dashed, rounded corners, inner sep=6pt},
  note/.style={align=left, font=\small}
}

\DeclareMathOperator*{\argmin}{arg\,min}

\usetikzlibrary{
  shapes,
  arrows.meta,
  positioning,
  fit,
  backgrounds,
  patterns,
  calc,
  decorations.pathreplacing
}

\usepackage{wrapfig} 
\usepackage{listings}  
\usepackage{xcolor}

\usepackage{adjustbox} 
\usepackage{subcaption} 

\lstdefinestyle{Pytorch}{
    language         = Python,
    backgroundcolor  = \color{white},
    basicstyle = \fontsize{8.0pt}{9pt}\selectfont\ttfamily\bfseries,
    columns          = fullflexible,
    breaklines       = true,
    captionpos       = b,
    commentstyle     = \fontsize{4pt}{4pt}\color{codeblue},
    keywordstyle     = \fontsize{4pt}{4pt}\color{codekw},
    morekeywords     = {with,scatter_,norm,sort},
}
\definecolor{codeblue}{rgb}{0.0, 0.0, 0.5} 
\definecolor{codekw}{rgb}{0.0, 0.0, 0.5} 

\definecolor{Gray}{gray}{0.95}

\lstdefinestyle{Pytorch}{
    language         = Python,
    backgroundcolor  = \color{white},
    basicstyle = \fontsize{8.0pt}{9pt}\selectfont\ttfamily\bfseries,
    columns          = fullflexible,
    breaklines       = true,
    captionpos       = b,
    commentstyle     = \fontsize{4pt}{4pt}\color{codeblue},
    keywordstyle     = \fontsize{4pt}{4pt}\color{codekw},
    morekeywords     = {with,scatter_,norm,sort},
}

\usepackage{algorithmic}
\usepackage{algorithm}

\newtheorem{proposition}{Proposition}

\theoremstyle{definition}

\theoremstyle{remark}

\setlist[itemize]{leftmargin=*, topsep=2pt, itemsep=2pt}
\setlist[enumerate]{leftmargin=*, topsep=2pt, itemsep=2pt}

\title{A GPU-accelerated Nonlinear Branch-and-Bound Framework for Sparse Linear Models}
\author{
  Ryan Lucas\thanks{Equal contribution.}\\MIT Operations Research Center\\\texttt{ryanlu@mit.edu}
  \and
  Xiang Meng\footnotemark[1]\\MIT Operations Research Center\\\texttt{mengx@mit.edu}
  \and
  Rahul Mazumder\\MIT Sloan School of Management, ORC, and MIT Center for Statistics\\\texttt{rahulmaz@mit.edu}
}
\date{}

\begin{document}
\maketitle

\begin{abstract}
We study exact sparse linear regression with an $\ell_0$--$\ell_2$ penalty and develop a branch-and-bound (BnB) algorithm that is explicitly designed for GPU execution. Starting from a perspective reformulation, we derive an interval relaxation that can be solved by ADMM with closed-form, coordinate-wise updates. We structure these updates so that the main work at each BnB node reduces to batched matrix--vector operations with a shared data matrix, enabling fine-grained parallelism across coordinates and coarse-grained parallelism across many BnB nodes on a single GPU. Feasible solutions (upper bounds) are generated by a projected gradient method on the active support, which is implemented in a batched fashion so that many candidate supports are updated in parallel on the GPU. We discuss practical design choices such as memory layout, batching strategies, and load balancing across nodes that are crucial for obtaining good utilization on modern GPUs. On synthetic and real high-dimensional datasets, our GPU-based approach achieves clear runtime improvements over a CPU implementation of our method, an existing specialized BnB method, and commercial MIP solvers.
\end{abstract}

\vspace{0.5em}
\noindent\textbf{Keywords:} Branch-and-Bound; Mixed Integer Programming; Sparse Learning; GPU

\section{Introduction} 

Mixed integer programming (MIP) and especially mixed integer nonlinear programming (MINLP), has emerged as a promising tool for modeling and computation of foundational statistics and machine learning (ML) problems that admit a combinatorial description \cite{bertsimas2019modern}. Common examples include problems such as high dimensional sparse regression \cite{hazimeh2021sparse, bertsimas2020sparse}, graphical models \cite{behdin2023sparsegaussiangraphicalmodels}, decision trees \cite{bertsimas2017optimal}, robust statistics (outlier detection) \cite{gómez2023outlierdetectionregressionconic}, among others \cite{tillmann2022cardinalityminimizationconstraintsregularization}. MIP-based estimators in high-dimensional statistics are known to enjoy optimal statistical properties in suitable regimes  \cite{wainwright2019high, zhang2014lower, Gamarnik2022, mazumder2022subset, behdin2021sparsepcanewscalable,hazimeh2021groupedvariableselectiondiscrete}.

In this paper we focus on a canonical problem from this family, the sparse linear regression with $\ell_2$ regularization. Given a dataset \((X, y) \in \mathbb{R}^{n \times p} \times \mathbb{R}^n\), where \(X\) is the design matrix containing \(n\) samples and \(p\) features, and \(y\) is the response vector, the $\ell_0$-$\ell_2$ penalized least squares problem is given by:
\begin{equation} \label{eq:original}
    \min _{\beta \in \mathbb{R}^p} \frac{1}{2}\|y-X \beta\|_2^2+\lambda_0\|\beta\|_0+\lambda_2\|\beta\|_2^2,
\end{equation}
where \(\beta\) denotes the regression coefficients, The \(\ell_0\)-(pseudo) norm \(\|\beta\|_0\) counts the number of nonzero elements in \(\beta\) and the squared \(\ell_2\)-norm \(\|\beta\|_2^2\) imposes a penalty on the magnitude of the coefficients. The regularization parameters \(\lambda_0 \geq 0\) and \(\lambda_2 \geq 0\) are hyperparameters selected by practitioners to control the trade-off between model fit, sparsity, and shrinkage. Specifically, \(\lambda_0\) controls the level of sparsity by penalizing the inclusion of less relevant features, while \(\lambda_2\) regulates the degree of shrinkage applied to the coefficients. Several works have explored the statistical properties of $\ell_0$-based estimators~\cite{greenshtein2006best,candes2013well,zhang2014lower,david2017high,wainwright2009information,dedieu2021learning}.  Under suitable assumptions, global solutions to problem \eqref{eq:original} achieve optimal support recovery \cite{fletcher2009necessary} and prediction error bounds \cite{raskutti2011minimax}. These strong guarantees generally do not extend to approximate solutions obtained from heuristic methods. Moreover, \cite{mazumder2022subset} demonstrates that the $\ell_2$ term helps mitigate overfitting in low signal-to-noise ratio (SNR) settings.

Despite its usefulness, solving problem \eqref{eq:original} poses substantial computational challenges as the problem is $\mathcal{NP}$-hard \cite{natarajan1995sparse}. Significant efforts have been devoted to developing exact algorithms based on MIP. Off-the-shelf MIP solvers reformulate the problem into a MIP framework and leverage powerful branch-and-bound (BnB) techniques inherent in commercial solvers like Gurobi and CPLEX. These MIP solvers can solve moderate-sized instances of \eqref{eq:original}, with  \(n \sim p \lesssim 10^3\), to global optimality, but struggle with larger instances~\cite{bertsimas2016best}. There have been several works designing specialized algorithms that are capable of handling larger-scale problems within a MIP framework \cite{hazimeh2021sparse, bertsimas2020sparse, han2025compact, atamturk2020rankoneconvexificationsparseregression}. Nevertheless, the computational cost of certifying global optimality often requires traversing a massive BnB search tree, which can result in large run-times. We aim to advance the computational practice of solving problem \eqref{eq:original} by developing methods to reduce the time required to certify optimality for large-scale problems. Different from earlier work in this area, we consider algorithmic development keeping in mind {\emph{hardware-algorithm co-design}}, especially in using GPUs to accelerate the overall BnB procedure, as discussed next.

Most existing algorithms for problem \eqref{eq:original} are designed to run on classical CPU architectures, while approaches designed to be run on more scalable computing hardware such as GPUs have been explored to a lesser extent (See section \ref{sec:related_work}). A GPU can perform on the order of tens of trillions of floating point operations per second, where a typical CPU does orders less \cite{sountsov2024runningmarkovchainmonte}. GPUs have been used extensively in scientific computing to accelerate compute-intensive workloads in fields such as atmospheric modeling \cite{Kelly2010} and physical simulations \cite{Navarro2014}, among others. GPU-based computing has also become the central paradigm powering the enormous success of deep learning approaches for high-dimensional unstructured data~\cite{MITTAL2019101635, Buber2018, pandey2022}. More recently, GPUs have also begun to be used in optimization, a notable example being specialized linear programming solvers that exploit  parallelism on GPUs~\cite{Jung2008,Spamp2009,swirydowicz2022linear,lu2024cupdlpjlgpuimplementationrestarted}. 
However, their implications beyond this family of problems especially non-linear BnB remain to be fully explored. Recently, \cite{liu2025scalable} propose a first-order method for solving the relaxed perspective formulation with a cardinality constraint to certify optimal k-sparse GLMs, leveraging GPU acceleration solely for computing the lower bound at a single node.

The holistic integration of GPUs into a BnB framework is non-trivial due to a fundamental mismatch between the hardware and traditional algorithms. GPUs are designed for massive parallelism, and excel when executing identical operations across large blocks of data, whereas BnB involves traversing a search tree in an irregular manner. To fully leverage GPU acceleration, in this paper we develop new algorithmic methodology integrating different components of a non-linear BnB framework that exploit parallel execution to solve Problem~\eqref{eq:original}. For the node relaxation algorithm, we propose an integrated approach centered on a specialized First-Order Method (FoM) designed for simultaneous variable updates at each node. Specifically, the updates within our FoM are  parallelizable across the elements of the decision variables, which leads to large speed-ups when executed on GPUs. A direct consequence of this FoM is the ability to compute primal solutions that are sparse (i.e., few nonzeros in $\beta$) and valid dual bounds. Furthermore, this algorithmic structure allows for the simultaneous solving of multiple subproblems, enabling parallelism \textit{across the BnB nodes}, leading to faster exploration of the tree and consequently faster dual bound certification. Finally, we design the algorithm so that the decision variables can be heavily warm-started, leading to significant reductions in per-node solve times as the tree is traversed. As illustrated in Figure~\ref{fig:gpu_bnb_overview}, we keep the lightweight BnB logic including node selection, branching, and global bound updates on the CPU, while offloading the computationally expensive lower- and upper-bound algorithms to the GPU.

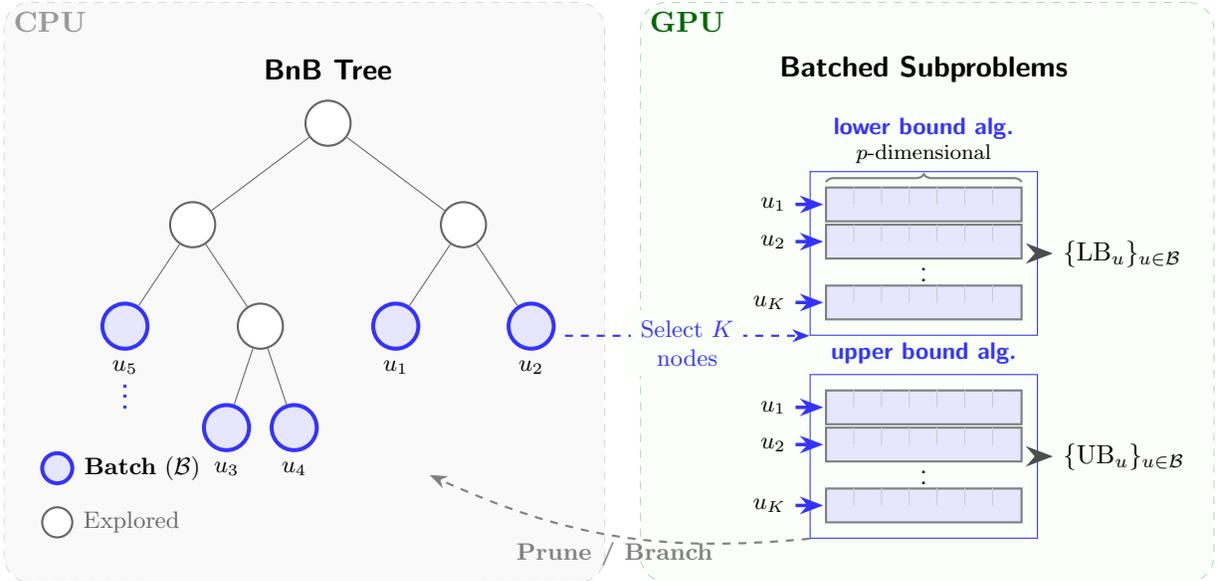
\begin{figure}[H]
\centering
\begin{adjustbox}{max width=\linewidth}
\begin{tikzpicture}[
    scale=0.85,
    font=\sffamily,
    >=Stealth,
    tnode/.style={circle, draw=black!60, thick, minimum size=0.6cm, inner sep=0pt, fill=white},
    active/.style={circle, draw=blue!80, line width=1.5pt, minimum size=0.6cm, inner sep=0pt, fill=blue!10},
    cpu_region/.style={draw=gray!30, dashed, fill=gray!5, rounded corners=10pt, inner sep=15pt},
    gpu_region/.style={draw=green!50!black!30, dashed, fill=green!2, rounded corners=10pt, inner sep=15pt},
    lb_row/.style={
        rectangle,
        draw=black!50,
        thick,
        fill=blue!10,
        minimum height=0.45cm,
        minimum width=2.6cm,
        inner sep=0pt,
        outer sep=0pt
    },
    ub_row/.style={
        rectangle,
        draw=black!50,
        thick,
        fill=blue!10,
        minimum height=0.45cm,
        minimum width=2.6cm,
        inner sep=0pt,
        outer sep=0pt
    },
    grid_line/.style={
        draw=black!20,
        thin
    },
    proc_arrow/.style={
        ->,
        line width=1.3pt,
        shorten >=1pt
    }
]

    \node[cpu_region, fit={( -4.2, 1.2) ( 3.5, -6.2)}] (cpu_box) {};
    \node[anchor=north west, text=gray!80, font=\bfseries] at (cpu_box.north west) {CPU};
    
    \node[gpu_region, fit={( 5.2, 1.2) ( 12.5, -6.2)}] (gpu_box) {};
    \node[anchor=north west, text=green!40!black, font=\bfseries] at (gpu_box.north west) {GPU};

    \node[tnode] (root) at (0,0) {};
    \node[tnode] (L1_1) at (-2, -1.5) {}; \node[tnode] (L1_2) at (2, -1.5) {};
    \draw[black!60] (root) -- (L1_1); \draw[black!60] (root) -- (L1_2);
    
    \node[active, label=below:\footnotesize $u_5$] (L2_1) at (-3, -3) {}; 
    \node[tnode]  (L2_2) at (-1, -3) {};
    
    \node[active, label=below:\footnotesize $u_1$] (L2_3) at (1, -3) {};
    \node[active, label=below:\footnotesize $u_2$] (L2_4) at (3, -3) {};
    
    \draw[black!60] (L1_1) -- (L2_1); \draw[black!60] (L1_1) -- (L2_2);
    \draw[black!60] (L1_2) -- (L2_3); \draw[black!60] (L1_2) -- (L2_4);
    
    \node[active, label=below:\footnotesize $u_3$] (L3_1) at (-1.5, -4.5) {};
    \node[active, label=below:\footnotesize $u_4$] (L3_2) at (-0.5, -4.5) {};
    
    \node[below=0.1cm of L2_1, text=blue!80] (tree_dots) {\Large $\vdots$};
    
    \draw[black!60] (L2_2) -- (L3_1); \draw[black!60] (L2_2) -- (L3_2);
    \node[align=center] at (0, 0.8) {\textbf{BnB Tree}};

    \begin{scope}[shift={(-4, -5.5)}]
        \node[active, minimum size=0.4cm] (l1) at (0, 0.4) {};
        \node[anchor=west, font=\footnotesize] at (l1.east) {\textbf{Batch} ($\mathcal{B}$)};
        \node[tnode, minimum size=0.4cm] (l2) at (0, -0.4) {};
        \node[anchor=west, font=\footnotesize, text=black!60] at (l2.east) {Explored};
    \end{scope}

    \node (stack_label) at (8.8, 0.8) {\textbf{Batched Subproblems}};

    \node[lb_row, label={[xshift=-0.4cm]left:\footnotesize $u_1$}] (Beta1) at (8.8, -1.2) {};
    \node[lb_row, label={[xshift=-0.4cm]left:\footnotesize $u_2$}] (Beta2) at (8.8, -1.75) {};
    \node (BetaDots) at (8.8, -2.2) {$\vdots$};
    \node[lb_row, label={[xshift=-0.4cm]left:\footnotesize $u_K$}] (BetaK) at (8.8, -2.65) {};

    \foreach \nodename in {Beta1, Beta2, BetaK} {
        \begin{scope}
            \clip (\nodename.south west) rectangle (\nodename.north east);
            \foreach \i in {1,...,6} {
                 \draw[grid_line] ($(\nodename.west)!0.142*\i!(\nodename.east)$) -- ++(0,1);
            }
        \end{scope}
        
        \draw[proc_arrow, blue!80]
            ($(\nodename.west) - (0.45, 0)$) -- (\nodename.west);
    }

    \node[
        draw=blue!80,
        fit=(Beta1)(BetaK),
        inner sep=6pt,
        label={[name=lbl_lb, yshift=3mm]above:\footnotesize\textcolor{blue!80}{\textbf{lower bound alg.}}}
    ] (MatrixBeta) {};
    
    \draw[decorate, decoration={brace, amplitude=3pt}, thick, gray]
        ($(Beta1.north west) + (0, 0.08)$) -- ($(Beta1.north east) + (0, 0.08)$)
        node[midway, above=3pt, font=\scriptsize, text=black] {$p$-dimensional};

    \node[ub_row, label={[xshift=-0.4cm]left:\footnotesize $u_1$}] (B1) at (8.8, -4.2) {};
    \node[ub_row, label={[xshift=-0.4cm]left:\footnotesize $u_2$}] (B2) at (8.8, -4.75) {};
    \node (BDots) at (8.8, -5.2) {$\vdots$};
    \node[ub_row, label={[xshift=-0.4cm]left:\footnotesize $u_K$}] (BK) at (8.8, -5.65) {};

    \foreach \nodename in {B1, B2, BK} {
        \begin{scope}
            \clip (\nodename.south west) rectangle (\nodename.north east);
            \foreach \i in {1,...,6} {
                 \draw[grid_line] ($(\nodename.west)!0.142*\i!(\nodename.east)$) -- ++(0,1);
            }
        \end{scope}
        
        \draw[proc_arrow, blue!80]
            ($(\nodename.west) - (0.45, 0)$) -- (\nodename.west);
    }

    \node[
        draw=blue!80,
        fit=(B1)(BK),
        inner sep=6pt,
        label={[name=lbl_ub]above:\footnotesize\textcolor{blue!80}{\textbf{upper bound alg.}}}
    ] (MatrixB) {};

    
    \draw[->, blue!80, dashed, thick]
         (3.5, 0 |- MatrixBeta.south west) -- (MatrixBeta.south west)
         node[midway, fill=green!2, align=center, font=\footnotesize, yshift=-0.1cm] {Select $K$\\nodes};

    \node[align=center, right=0.2cm of MatrixBeta, font=\small] (out_LB) {$\{\text{LB}_{u}\}_{u \in \mathcal{B}}$};
    \node[align=center, right=0.2cm of MatrixB, font=\small] (out_UB) {$\{\text{UB}_{u}\}_{u \in \mathcal{B}}$};

    \draw[->, ultra thick, black!70] (MatrixBeta.east) -- (out_LB.west);
    \draw[->, ultra thick, black!70] (MatrixB.east) -- (out_UB.west);

    \draw[->, thick, gray, dashed] (MatrixB.south west) to[bend left=15]
        node[midway, below, align=center, font=\footnotesize, inner sep=2pt] {\textbf{Prune / Branch}}
        (1.5, -5.2);

\end{tikzpicture}
\end{adjustbox}
\caption{High-level view of our GPU-accelerated BnB algorithm. A batch $\mathcal{B}$ of $K$ active nodes from the CPU-side BnB tree is mapped to batched node subproblem algorithms on the GPU. Each node $u \in \mathcal{B}$ is solved in parallel and the resulting lower ($\text{LB}_u$) and upper bounds ($\text{UB}_u$) are fed back to guide branching and pruning decisions on CPU.}
\label{fig:gpu_bnb_overview}
\end{figure}

\subsection{Summary of Contributions} Our contributions in the paper are as follows:
\begin{itemize}
\item At the core of our BnB framework is a node relaxation algorithm that is highly amenable to parallel computation. By introducing auxiliary and dual variables, we decompose each BnB node subproblem into components that can be solved independently via the Alternating Direction Method of Multipliers (ADMM) algorithm \cite{boyd2011}. Using problem-structure, we choose the ADMM splitting such that the solution to each of the resulting ADMM subproblems are computed in closed-form and highly parallel, meaning they can be accelerated on GPUs.

\item We leverage the separability of our node relaxation algorithm to parallelize subproblems across nodes in the BnB tree. Since problem (\ref{eq:original}) can have an extremely large number of binary variables, this strategy leads to dramatic speed-ups compared to solving nodes sequentially.

\item We introduce an upper bound algorithm (primal heuristic) based on projected gradient descent that quickly generates high-quality feasible solutions. To fully leverage GPU acceleration, this algorithm is also designed to be parallel both within a node and across nodes in the BnB tree. 

\item Our empirical evaluation demonstrates that our GPU-accelerated BnB algorithm strongly outperforms both commercial solvers and prior custom methods.  On large instances (e.g. $n=10^4$, $p=10^5$), even without node parallelism, our approach GPU-based method achieves a $35\times$ speedup over our CPU implementation and a $6\times$ speedup over L0BnB~\cite{hazimeh2021sparse}; at this scale, commercial solvers exceed memory limits. Adding node parallelism on top can lead to an up to 25x increase in nodes solved per second, corresponding to much faster solving times.

\end{itemize}


\subsection{Related work}\label{sec:related_work}

There have been various MIP-based approaches to solve problem \eqref{eq:original} as well as the cardinality-constrained version. \cite{bertsimas2016best} solve the cardinality-constrained problem to optimality using off-the-shelf MIP solver Gurobi. \cite{bertsimas2020sparse} proposed an outer-approximation approach for Problem~\eqref{eq:original} 
by constructing piece-wise lower bounds to the convex objective. Using Gurobi to solve the integer linear sub-problems, this scales to $p \approx 10^4$--$10^5$ when $n$ and $\lambda_2$ are sufficiently large. \cite{hazimeh2020fast} propose scalable algorithms to generate high-quality feasible solutions for large instances of Problem~\eqref{eq:original}. \cite{hazimeh2021sparse} present an integrated BnB framework for Problem~\eqref{eq:original} using first-order methods (coordinate descent) to solve the node relaxations. They can optimally solve instances with $n \approx 10^3$, $p \approx 10^5$ or larger values of $p$. Our approach is in the same algorithmic vein: we retain a BnB-based strategy for certifying optimality, but reformulate the per-node computations to enable parallel execution on GPUs both within node subproblems and across multiple node subproblems, significantly accelerating large-scale instances compared to prior work.
\cite[Ch.~5]{chen_2023} presents an extension of~\cite{hazimeh2021sparse} to $\ell_0$-$\ell_2$-penalized generalized linear models (GLMs) such as least squares and logistic loss.  
Recently, Liu et al.~\cite{liu2025scalable} extend the BnB framework of~\cite{hazimeh2021sparse} to consider GLMs with cardinality constraints. To solve the node relaxations, \cite{liu2025scalable} use accelerated proximal-gradient descent and leverage GPU acceleration (they consider GPU-acceleration within a node and not across the nodes).  A complementary line of work develops stronger convex relaxations for related mixed-integer quadratic formulations with indicator variables \cite{atamturk2020rankoneconvexificationsparseregression, bhathena2024parametricapproachsolvingconvex}.

The use of GPUs more broadly in optimization has gained significant traction, from enhancing local search subroutines or other heuristics  within serial algorithms 
 \cite{SCHULZ2013159} to the development of novel solvers explicitly designed to exploit parallel processing capabilities~\cite{anzt2014optimizing, tomov2010dense, Abbas_2022_CVPR, pacaud2023, shin2024acceleratingoptimalpowerflow}. One area where GPUs are gaining popularity are linear programming solvers \cite{Jung2008, Spamp2009, lu2024cupdlpjlgpuimplementationrestarted, swirydowicz2022linear}. Traditional linear programming algorithms, such as the simplex method, have been highly optimized for serial execution on CPUs, and do not naturally exploit parallelism. Recent work studies alternatives to fast serial algorithms with new parallel algorithms that can better exploit GPU processing capabilities \cite{lu2024cupdlpjlgpuimplementationrestarted}. Similarly, in convex optimization, conic solvers using interior-point methods are typically the go-to general-purpose algorithm, but pose challenges due to the memory requirements of solving the KKT system on GPUs \cite{pacaud2023}. Moreover, within BnB algorithms, thousands or millions of subproblems may need to be solved, and warm-starting becomes a crucial strategy but is not always exploitable via methods like interior-point algorithms \cite{anders2005}. Here, there is clearly an interplay emerging between algorithmic development and the strengths and weaknesses of the target hardware.

While parallel implementations of BnB algorithms have been studied since the 1990s and before \cite{Clausen1997}, these approaches have largely focused on coarse-grained parallelism using either shared-memory or message-passing systems. Early attempts at fine-grained parallelism were limited by synchronization overhead and memory constraints, making them suitable only for small problem instances \cite{Gendron94}. Our work represents (to our knowledge) the first parallel BnB implementation that effectively exploits modern GPU architectures' capabilities, enabling both parallel subproblem solving and parallel node exploration at scale.

\section{Preliminaries: MIP Formulation and Branch-and-Bound Solver}
We present MIP formulation of problem \eqref{eq:original}, and a sketch of our BnB framework.

\subsection{Perspective formulation with Big-M constraint }
We consider a MIP formulation of problem \eqref{eq:original}, incorporating Big-M constraints and a perspective reformulation
\cite{hazimeh2021sparse,akturk2009strong}. 

We define a constant $M > 0$ such that an optimal solution $\beta^*$ of problem \eqref{eq:original} satisfies $\|\beta^*\|_{\infty} \leq M$. Binary variables $z_i, i \in [p]$, are introduced to control whether each $\beta_i$ is zero, enforced through the constraints $-M z_i \leq \beta_i \leq M z_i, \, i \in [p]$. To reformulate the ridge term $\|\beta\|_2^2$, we introduce auxiliary continuous variables $s_i \geq 0$ and impose rotated second-order cone constraints $\beta_i^2 \leq s_i z_i$ for each $i \in [p]$. The term $\lambda_2\|\beta\|_2^2$ is then replaced with $\lambda_2\sum_{i \in [p]} s_i$. This results in the following perspective reformulation:
\begin{equation} \label{eq:perspective}
\begin{aligned}
 \min _{\beta, z, s} \quad & \frac{1}{2}\|y-X \beta\|_2^2+\lambda_0 \sum_{i \in[p]} z_i+\lambda_2 \sum_{i \in[p]} s_i \\
\text { s.t. } \quad &  -M z_i \leq \beta_i \leq M z_i, ~~ i \in[p], \\
&\beta_i^2 \leq s_i z_i,~~~\,\qquad\qquad i \in[p], \\
& z_i \in\{0,1\}, s_i \geq 0, \,\,~~~~i \in[p].
\end{aligned}
\end{equation}
As long as $M$ is selected appropriately, an optimal solution to the perspective formulation \eqref{eq:perspective} is also an optimal solution to the original problem \eqref{eq:original}. Practical approaches for estimating $M$ can be found in \cite{bertsimas2016best,xie2020scalable}. Problem \eqref{eq:perspective}, a mixed-integer second-order conic problem, becomes a second-order cone problem when each binary variable is relaxed to $z_{i} \in [0, 1]$ for all $i$.


\subsection{Overview of our Custom Branch-and-Bound framework}  \label{subsect:bnb}

We present an overview of our BnB framework for Problem \eqref{eq:perspective}. We identify parts of our BnB framework that align with traditional BnB, and then design our subproblem solvers to take advantage of parallelism in the BnB tree. Section \ref{sect:gpubnb} presents more detail in designing each component to be accelerated with GPUs. 

\subsubsection{Traditional Branch-and-Bound}

The core principle of BnB is to systematically explore the solution space by dividing it into smaller subproblems (branching) and eliminating parts of the search space that cannot contain the optimal solution (bounding). We sketch the components of our algorithm below, resembling a typical BnB procedure:

\begin{itemize}
    \item \textbf{Binary branching.}
    For problem \eqref{eq:perspective}, each subproblem of the solution space can be represented by two sets, $(\mathcal{F}_0, \mathcal{F}_1)$, which correspond to solving the problem with additional constraints: $z_i = 0$ for all $i \in \mathcal{F}_0$ and $z_i = 1$ for all $i \in \mathcal{F}_1$. The BnB framework explores the solution space using a binary tree structure. It begins at the root node, where the subproblem is $(\mathcal{F}_0, \mathcal{F}_1) = (\varnothing, \varnothing)$. For branching, the algorithm selects a variable, say $j \in [p]$, for the current node and creates two new nodes: one with $(\mathcal{F}_0 \cup \{j\}, \mathcal{F}_1)$ and another with $(\mathcal{F}_0, \mathcal{F}_1 \cup \{j\})$.

    \item \textbf{Node relaxation algorithm (lower bounds).}
    At each node, we solve an interval relaxation subproblem to obtain a lower bound on the best achievable objective within that node. This relaxation provides the dual (lower) bound used to decide whether a subproblem and its descendants can be safely discarded.

    \item \textbf{Upper-bound algorithm (primal solutions).}
    An upper-bound solver quickly computes an approximate feasible solution for the node, which provides an upper bound to problem \eqref{eq:perspective} as a whole. These feasible solutions update the best-known objective value and thus tighten the global upper bound.

    \item \textbf{Pruning and search strategy.}
    To reduce the search tree size, BnB prunes nodes under two conditions: (i) if the lower-bound solver at the current node produces an integer solution for $z$ or (ii) if the node's lower bound exceeds the best-known upper bound. While various node selection strategies can be used (e.g., breadth-first search or depth-first search), we adopt a best-first node selection strategy \cite{linderoth1999computational} to accelerate the search further---at each step, BnB selects the node or nodes with the smallest lower bound to solve.

    \item \textbf{Global lower bound and termination.}
    The global lower bound for the BnB tree is the smallest lower bound among all open leaf nodes. As the algorithm progresses, the global lower bound converges to the optimal objective of Problem~\eqref{eq:perspective}. In practice, the algorithm can be terminated early if the gap between the best upper and lower bounds falls below a predefined threshold.
\end{itemize}
\subsection{Parallel Node Subproblem Solver}

At each node in the BnB tree, we compute both an upper bound and a lower bound relaxation of problem \eqref{eq:perspective}. These are typically the most computationally intensive portions of any BnB algorithm, being the great majority of the overall runtime. Our framework employs two specialized algorithms designed for parallel computation: an ADMM-based lower bound algorithm and a projected gradient descent based upper bound algorithm. These algorithms are structured to enable parallel computation across the coordinates of the optimization variables, making them highly efficient on GPU architectures. The ADMM algorithm, detailed in Section \ref{subsect:ADMM}, decomposes the problem into subproblems that can be solved in parallel on GPUs. Similarly, the upper bound algorithm detailed in Section \ref{sec: upper_bound} operates on the entire support simultaneously rather than cycling through the decision variables.

\subsection{Parallel Tree Exploration}

As noted in before, our ADMM-based node relaxation and proximal-gradient upper bound algorithms are explicitly designed to support parallelism within a subproblem. However, a key feature of our approach is that the very same structure also enables parallelism across nodes in the BnB tree. Let $\mathcal{N}$ denote the set of open (unexplored) nodes and let $\mathrm{UB}$, $\mathrm{LB}$ denote the current global upper and lower bounds. At each iteration, we select a batch:
\[
\mathcal{B} \subseteq \mathcal{N}, \quad |\mathcal{B}| = \min(K, |\mathcal{N}|),
\]
consisting of up to $K$ nodes. For every node $u \in \mathcal{B}$, defined by fixings $(\mathcal{F}_0^{(u)}, \mathcal{F}_1^{(u)})$, we compute in parallel:
\begin{enumerate}
    \item a dual bound $\mathrm{LB}_u$ via the batched ADMM relaxation (Sections~\ref{subsect:ADMM_parallel}),
    \item a feasible solution $\beta_u^{\mathrm{feas}}$ with cost $\mathrm{UB}_u$ via the batched projected gradient method (Section~\ref{sec:upper_bound_parallel}).
\end{enumerate}
The global bounds are then updated as:
\[
\mathrm{UB} \leftarrow \min\bigl\{\mathrm{UB},\; \min_{u \in \mathcal{B}}\, \mathrm{UB}_u\bigr\}, \qquad \mathrm{LB} \leftarrow \min_{u \in \mathcal{N}}\, \mathrm{LB}_u,
\]
and every node $u$ satisfying $\mathrm{LB}_u \geq \mathrm{UB}$ is pruned from $\mathcal{N}$. Remaining nodes are branched: for a chosen index $j \in [p] \setminus (\mathcal{F}_0^{(u)} \cup \mathcal{F}_1^{(u)})$, two children $(\mathcal{F}_0^{(u)} \cup \{j\},\, \mathcal{F}_1^{(u)})$ and $(\mathcal{F}_0^{(u)},\, \mathcal{F}_1^{(u)} \cup \{j\})$ are added to $\mathcal{N}$. The algorithm terminates when $\mathcal{N} = \emptyset$ or $(\mathrm{UB} - \mathrm{LB})/\mathrm{UB} \leq \mathrm{tol}$. The full procedure is summarized in Algorithm~\ref{algo:BnB-approx}.


\begin{algorithm}[H]
\small
\begin{algorithmic}[1]
\REQUIRE Batch size $K$, inputs; \textbf{Init} $\mathcal{N} \leftarrow \{ u_0 \}$, $\text{UB} \leftarrow \infty, \text{LB} \leftarrow -\infty$
\WHILE{ $\mathcal{N} \neq \emptyset$ and $(\text{UB} - \text{LB}) / \text{UB} > \text{tol}$ }
    \STATE {\color{green!30!blue} \textbf{Select nodes $\mathcal{B} \subseteq \mathcal{N}$ ($\min(K, |\mathcal{N}|)$) \& solve for $\{\text{LB}_u, \text{UB}_u, \beta^{\text{feas}}_u \}_{u \in \mathcal{B}}$ in parallel}}
    \FOR{each $u \in \mathcal{B}$}
        \STATE \textbf{if} $\text{LB}_u \geq \text{UB}$ \textbf{then} Prune node $u$; \textbf{continue}
        \STATE \textbf{if} $\text{UB}_u < \text{UB}$ \textbf{then} $\text{UB} \leftarrow \text{UB}_u, \;\beta^* \leftarrow \beta^{\text{feas}}_u$
        \STATE Select $j \in [p] \backslash (\mathcal{F}_0 \cup \mathcal{F}_1)$; Add $(\mathcal{F}_0 \cup \{j\}, \mathcal{F}_1, \text{LB}_u)$ and $(\mathcal{F}_0 , \mathcal{F}_1 \cup \{j\}, \text{LB}_u)$ to $\mathcal{N}$
    \ENDFOR
    \STATE Update global lower bound $\text{LB} \leftarrow \min\{\text{LB}_u : u \in \mathcal{N}\}$
\ENDWHILE
\RETURN $\beta^*$
\end{algorithmic}
\caption{\small Overview of our parallel BnB solver. The computationally intensive node subproblems are solved in parallel on GPU, indicated in \textbf{{\color{green!30!blue}blue}}, while the lightweight serial BnB logic on CPU remains unchanged.}
\label{algo:BnB-approx}
\end{algorithm}

\section{Branch-and-Bound Algorithm Components}\label{sect:gpubnb}

In this section, we describe the various components of our BnB algorithm in detail. Our goal is to design algorithms that are highly parallel, in contrast to the more serial algorithms introduced in prior work. This approach allows us to leverage the computational power of modern GPUs for finding optimal or near-optimal solutions to problem \eqref{eq:original}. To this end, we develop two key components:

\begin{itemize}
    \item An ADMM-based node relaxation algorithm which fully separates across the elements of the decision variables for each subproblem and across subproblems. A key feature in our approach is that these two forms of parallelism are integrated.

    \item A proximal gradient based upper bound algorithm designed to operate over the entire support at a given node, rather than cycling through the decision variables sequentially.
\end{itemize}

We first describe how parallelism can be implemented within a node, via our ADMM-based node relaxation and proximal gradient upper-bound algorithms (Sections~\ref{subsect:ADMM} and~\ref{sec: upper_bound}), and then how to extend to parallelism across nodes in the BnB tree (Sections~\ref{subsect:ADMM_parallel} and~\ref{sec:upper_bound_parallel}).

\subsection{ADMM-based algorithm for the node relaxations}\label{subsect:ADMM}

We describe an ADMM method for solving the relaxation Problem~\eqref{eq:perspective}.
We first consider a particular node, while parallelism across nodes is discussed in Section~\ref{subsect:ADMM_parallel}. At a given node of the BnB tree whose corresponding upper and lower bounds are $(\mathcal{F}_0,\mathcal{F}_1)$, the interval relaxation of the subproblem reads:
\begin{equation} \label{eq:relaxnode}
\begin{aligned}
 \min _{\beta, z, s} \quad & \frac{1}{2}\|y-X \beta\|_2^2+\lambda_0 \sum_{i \in[p]} z_i+\lambda_2 \sum_{i \in[p]} s_i \\
\text { s.t. } \quad & \beta_i^2 \leq s_i z_i,\,\,\, i \in[p] \\
& -M z_i \leq \beta_i \leq M z_i, \,\,\, i \in[p] \\
& z_i \in [0,1], s_i \geq 0, \,\,\,i \in[p] \\
& z_i=0\, \forall \,i \in \mathcal{F}_0,\,\, z_i=1\,\forall \,i \in \mathcal{F}_1.
\end{aligned}
\end{equation}
To simplify the problem, we use the fact that for a fixed $\beta$, the optimal values of $s$ and $z$ can be computed analytically (see Appendix \ref{sect:app-technical} for derivations following \cite{hazimeh2021sparse}). Substituting the optimal values back into the objective function in terms of $\beta$, we obtain an equivalent problem that depends only on $\beta$:
\begin{equation} \label{eq:relaxnode2}
\begin{aligned}
\min _{\beta \in \mathbb{R}^p} &\,\, \frac{1}{2}\|y-X \beta\|_2^2+\sum_{i \in[p]} \psi_i\left(\beta_i ; \lambda_0, \lambda_2, M\right) \\
\text { s.t. } &\,\,  \|\beta\|_{\infty} \leq M
\end{aligned}
\end{equation}
where  
\begin{equation} \label{eq:psi}
\begin{aligned}
\psi_i\left(\beta_i ; \lambda_0, \lambda_2, M\right)= \begin{cases} \infty\cdot \textbf{1}_{\{\beta_i=0\}} & \text { if } i \in \mathcal{F}_0\\
\lambda_0 + \lambda_2 \beta_i^2 & \text { else if } i \in \mathcal{F}_1 \text{ or } \sqrt{\frac{\lambda_0}{\lambda_2}} \le |\beta_i| \le M\\
2\sqrt{\lambda_0\lambda_2} |\beta_i| & \text { else if } |\beta_i| \le \sqrt{\frac{\lambda_0}{\lambda_2}} \leq M \\ \left(\lambda_0 / M+\lambda_2 M\right)\left|\beta_i\right| & \text { 
 else if } \sqrt{\frac{\lambda_0}{\lambda_2}}>M\end{cases}
\end{aligned}
\end{equation}
Solving problem \eqref{eq:relaxnode2}  using ADMM involves a choice of splitting technique;  that is, a choice of auxiliary variable to decouple the problem. We reformulate problem \eqref{eq:relaxnode2} by introducing a copy $b$ of the primal variable $\beta$: 
\begin{equation} \label{eq:ADMM1}
\begin{aligned}
\min _{b, \beta \in \mathbb{R}^p} &\,\, \frac{1}{2}\|y-X b\|_2^2+\sum_{i \in[p]} \psi_i\left(\beta_i ; \lambda_0, \lambda_2, M\right)  + \infty \cdot \mathbf{1}_{\left\{\|\beta\|_{\infty} \leq M\right\}}\\
\text { s.t. } &\,\,b=\beta
\end{aligned}
\end{equation}

Here there are various splittings that could have been used. For example, a popular splitting is to introduce an auxiliary variable $r = y - X\beta$ (c.f.~\cite{boyd2011}, pg 39). However, this leads to coupled updates involving $X\beta$ that cannot be solved completely in parallel. By decoupling $b$ and $\beta$, we separate the smooth quadratic term involving $b$ from the non-smooth regularization term involving $\beta$, such that both resulting subproblems can be parallelized. The constraint $b = \beta$ ensures consistency between the variables.  ADMM \cite{boyd2011,davis2016convergence} is a powerful iterative first-order optimization method well-suited for problems with such a structure. It considers the augmented Lagrangian function of this problem that enforces the equality constraint using a dual variable $v \in \mathbb{R}^p$ and a penalty parameter $\rho > 0$:
\begin{equation} \label{eq:ADMM2}
\begin{aligned}
    \min _{b, \beta \in \mathbb{R}^p}  L_\rho(b,\beta,v)&= \frac{1}{2}\|y-X b\|_2^2+\sum_{i \in[p]} \psi_i\left(\beta_i ; \lambda_0, \lambda_2, M\right)\\
    & + \infty \cdot \mathbf{1}_{\left\{\|\beta\|_{\infty} \leq M\right\}}+ v^\top(b-\beta) +\frac{\rho}{2}\|b-\beta\|_2^2 
\end{aligned}
\end{equation}

The ADMM algorithm proceeds by iteratively minimizing the augmented Lagrangian with respect to primal variables \(b\) and \(\beta\), and dual variable \(v\). Specifically, at iteration \(t\), the updates are:
\begin{equation}\label{eq:admm-update}
    \begin{aligned}
        b^{(t+1)} &= \argmin_{b} L_\rho( b, \beta^{(t)}, v^{(t)})\\
        \beta^{(t+1)} &= \argmin\nolimits_{\beta} L_\rho( b^{(t+1)}, \beta, v^{(t)})  \\
       v^{(t+1)} &= v^{(t)} + \rho( b^{(t+1)}- \beta^{(t+1)}).
    \end{aligned}
\end{equation}

\subsubsection{Structuring the ADMM updates to exploit parallelism}
We describe how each subproblem can be solved in a separable fashion.

\noindent \textbf{Update for $b$.} To update of $b^{(t+1)}$ in \eqref{eq:admm-update}, we solve the optimization problem: \begin{align}\label{eq
} b^{(t+1)} & = \argmin_{b} \left\{ \frac{1}{2}\| y - Xb \|_2^2 + v^{(t)\top}(b - \beta^{(t)}) + \frac{\rho}{2} \| b - \beta^{(t)} \|_2^2 \right\} \\
& = \left(X^\top X+\rho I\right)^{-1}\left( X^\top y + \rho \beta^{(t)} -v^{(t)} \right).
\end{align} 
At first glance this update appears to couple the coordinates of $b$. However, notice that we can precompute the matrix $D = \left( X^\top X + \rho I \right)^{-1}$ and the vector $c = X^\top y$ before the ADMM iterations begin. When $p$ is much larger than $n$ (e.g. $p \sim 10^5$, $n \sim 10^3$), inverting this matrix can be expensive. However, we can exploit the equivalent reformulation of the inverse:
\[
D = 
\begin{cases}
\left( X^\top X + \rho I_p \right)^{-1} & \text{if } n \geq p, \\
\frac{1}{\rho} I_p - \frac{1}{\rho^2} X^\top \left( X X^\top + \rho I_n \right)^{-1} X & \text{if } p > n.
\end{cases}
\]

When \( X^\top X \) is a large matrix of size \( p \times p \), it is more efficient to compute $D$ to use the Sherman-Morrison-Woodbury formula~\cite{hager1989}, as it only involves inverting an $n\times n$ matrix (which can be much smaller than $p$ for sparse learning). Moreover, efficient matrix factorization routines reduce the inversion problem to solving two triangular systems via forward and backward substitution, each with $\mathcal{O}(n^2)$ complexity. Importantly, this precomputation of \(D\) only needs to be done once for the \textit{entire tree} in BnB. Once $D$ and $c$ are precomputed, the per-iteration computation of $b^{(t+1)}$ reduces to operations involving vectors, and the update can be parallelized. Specifically, for each coordinate $i$:
\begin{equation}\label{eq:b_update} b_i^{(t+1)} = \sum_{j=1}^p D_{ij} \left( c_j + \rho \beta_j^{(t)} - v_j^{(t)} \right), \quad \text{for } i = 1, \dotsc, p \end{equation}
hence, each coordinate $b_i$ can be assigned a unique thread and computed independently across the coordinates. 

\noindent \textbf{Update for $\beta$.} The update for $\beta^{(t+1)}$ in \eqref{eq:admm-update}
is inherently separable and can be computed in closed form for each coordinate $i\in [p]$ as:

\begin{equation}\label{eq:minbetalower}
    \begin{aligned}
    \beta_i^{(t+1)}  &  = \argmin_{\beta_i} \frac{\rho}{2} (\beta_i - \tilde{\beta_i})^2 + \psi_i\left(\beta_i ; \lambda_0, \lambda_2, M\right) + \infty \cdot \mathbf{1}_{\left\{|\beta_i| \leq M\right\}} \\
        & = \begin{cases} 0 & \text { if } i \in \mathcal{F}_0\\
T\left(\frac{\rho}{\rho+2\lambda_2}\tilde{\beta_i},0,M\right) & \text { else if } i \in \mathcal{F}_1 \text{ or }  \frac{2\sqrt{\lambda_0 \lambda_2}}{\rho}+\sqrt{\frac{\lambda_0}{\lambda_2}} \le |\tilde{\beta}_i| \\
T\left(\tilde{\beta}_i ; \frac{2\sqrt{\lambda_0 \lambda_2}}{\rho}, M\right) & \text { else if }|\tilde{\beta}_i| \leq \frac{2\sqrt{\lambda_0 \lambda_2}}{\rho}+\sqrt{\frac{\lambda_0}{\lambda_2}}, \sqrt{\frac{\lambda_0}{\lambda_2}}\le M\\
T\left(\tilde{\beta}_i ; \frac{\lambda_0}{M\rho}+\frac{\lambda_2 M}{\rho}, M\right) & \text { else if } \sqrt{\frac{\lambda_0}{\lambda_2}}>M
\end{cases}
\end{aligned}
\end{equation}
where $\tilde{\beta_i}=b_i + v_i / \rho$ and $T:\mathbb{R}\to\mathbb{R}$ is the box-constrained soft-thresholding operator:
\begin{equation}\label{eq:Tdef}
    \begin{aligned}
T(t ; a, m):= \begin{cases}0 & \text { if }|t| \leq a \\ (|t|-a) \operatorname{sign}(t) & \text {if } a<|t| \leq a+m \\ m \operatorname{sign}(t) & \text { otherwise. }\end{cases}
    \end{aligned}
\end{equation}

The solution in (\ref{eq:minbetalower}) comes from reasoning about the different cases of the piece-wise function in (\ref{eq:psi}) -- see Appendix \ref{sect:app-technical}. The coordinate is either zero (if $i$ belongs to the upper bound set $\mathcal{F}_0$) or is the result of a thresholding operation. These thresholding operations can be batched efficiently across the coordinates in popular libraries for GPU-accelerated computation, such as PyTorch. See Algorithm \ref{alg:beta-update-pytorch} for an example of how the $\beta$ update can be performed in PyTorch.

\begin{algorithm}[t]
\caption{PyTorch code for \(\beta\)-update}
\label{alg:beta-update-pytorch}
{\footnotesize
\begin{lstlisting}[style=Pytorch, mathescape=true, aboveskip=2pt, belowskip=2pt]
def $\beta_{\text{update}}$($\tilde{\beta}$, $\lambda_0$, $\lambda_2$, $\rho$, $M$, $\boldsymbol{a}$, $\mathcal{F}_0$, $\mathcal{F}_1$):
    return torch.where(
        $\mathcal{F}_0$, torch.zeros_like($\tilde{\beta}$),
        torch.where(
            $\mathcal{F}_1$,
            (($\rho$/($\rho$+2*$\lambda_2$))*$\tilde{\beta}$).clamp(-$M$, $M$),
            (torch.sign($\tilde{\beta}$) * (torch.abs($\tilde{\beta}$) - $\boldsymbol{a}$).clamp(min=0.0)).clamp(-$M$, $M$)
        )
    )
\end{lstlisting}
}
\end{algorithm}

\smallskip 


\noindent \textbf{Update for $v$}. Finally, we can also see that the update for $v$ in \eqref{eq:admm-update} is separable across the coordinates because each component \(v_i\) depends only on the corresponding components \(b_i\) and \(\beta_i\):
\begin{equation}\label{eq:v_i-update}
v_i^{(t+1)} = v_i^{(t)} + \rho \left( b_i^{(t+1)} - \beta_i^{(t+1)} \right), \quad \text{for } i = 1, \dotsc, p
\end{equation}
where, similar to the update for $b$, each coordinate of $v$ can be assigned a unique thread and computed in parallel. Overall, these updates align well with the Same Instruction Multiple Data (SIMD) programming model of GPUs, where many parallel elements are processed independently using the same program \cite{Owens2008}. \autoref{fig:parallel_comp} illustrates the ADMM updates used in our framework. As shown in the figure, ADMM allows for operator splitting by separating the optimization problem into sequential updates of primal variables \( b \), auxiliary variables \( \beta \), and dual variables \( v \). These updates can be computed in closed-form and decomposed across the elements of each vector for fast parallel computation--see \autoref{fig:parallel_comp} (right). 

 \vspace{-1mm}
 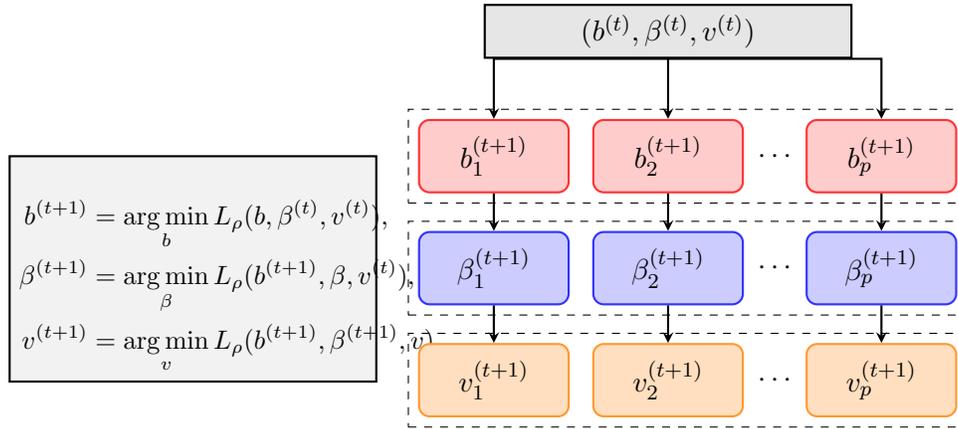
\begin{figure}[H]
\centering
\resizebox{0.88\linewidth}{!}{%
\begin{tikzpicture}[
    scale=1.0,
    node distance=1.5cm,
    auto,
    >=stealth,
    bnode/.style = {draw, rectangle, thick, fill=red!20, text width=4.5em, text centered, rounded corners, minimum height=2.5em, draw=red!80},
    betanode/.style = {draw, rectangle, thick, fill=blue!20, text width=4.5em, text centered, rounded corners, minimum height=2.5em, draw=blue!80},
    vnode/.style = {draw, rectangle, thick, fill=orange!25, text width=4.5em, text centered, rounded corners, minimum height=2.5em, draw=orange!80},
    line/.style = {thick, ->},
    dashed_line/.style = {thick, dashed, ->},
    branch/.style={circle, fill=black, minimum size=4pt, inner sep=0pt, outer sep=0pt}
    ]

  \node [rectangle, draw, thick, fill=gray!10, text width=12em, text width=15em, align=left] (equations) at (-4, 0.5) {
  \small
    \begin{align*}
        b^{(t+1)} &= \argmin_{b} L_\rho( b, \beta^{(t)}, v^{(t)}), \\
        \beta^{(t+1)} &= \argmin_{\beta} L_\rho( b^{(t+1)}, \beta, v^{(t)}),  \\
        v^{(t+1)} &= v^{(t)} + \rho\bigl(b^{(t+1)}-\beta^{(t+1)}\bigr)
    \end{align*}
  };

\node [bnode] (b1) at (0.7, 2) {\(b_1^{(t+1)}\)};
  \node [bnode, right=0.3cm of b1] (b2) {\(b_2^{(t+1)}\)};
  \node [right=0.05cm of b2] (bdots) {\(\cdots\)};
  \node [bnode, right=0.05cm of bdots] (bp) {\(b_p^{(t+1)}\)};

  \node [betanode, below=0.5cm of b1] (beta1) {\(\beta_1^{(t+1)}\)};
  \node [betanode, below=0.5cm of b2] (beta2) {\(\beta_2^{(t+1)}\)};
  \node [right=0.05cm of beta2] (betadots) {\(\cdots\)};
  \node [betanode, below=0.5cm of bp] (betap) {\(\beta_p^{(t+1)}\)};

  \node [vnode, below=0.5cm of beta1] (v1) {\(v_1^{(t+1)}\)};
  \node [vnode, below=0.5cm of beta2] (v2) {\(v_2^{(t+1)}\)};
  \node [right=0.05cm of v2] (vdots) {\(\cdots\)};
  \node [vnode, below=0.5cm of betap] (vp) {\(v_p^{(t+1)}\)};

  \draw [line] (beta1.south) -- (v1.north);
  \draw [line] (beta2.south) -- (v2.north);
  \draw [line] (betap.south) -- (vp.north);

  \draw [line] (b1.south) -- (beta1.north);
  \draw [line] (b2.south) -- (beta2.north);
  \draw [line] (bp.south) -- (betap.north);

  \node [rectangle, draw, thick, fill=gray!20, text width=12em, align=center, above=0.8cm of b2] (bupdate) {($b^{(t)}, \beta^{(t)}, v^{(t)}$)};

  \draw [line] (bupdate.south) -- (b2.north);
  \draw [line] (bupdate.south) -| (b1.north);
  \draw [line] (bupdate.south) -| (bp.north);

  \node [dashed, draw, fit=(b1)(b2)(bp), label=above:{}] {};
  \node [dashed, draw, fit=(beta1)(beta2)(betap), label=above:{}] {};
  \node [dashed, draw, fit=(v1)(v2)(vp), label=above:{}] {};

\end{tikzpicture}%
}
\caption{Illustration of the ADMM updates. \textbf{Left:} ADMM allows for operator splitting by separating the optimization problem into sequential updates of primal variables \(b\), auxiliary variables \(\beta\), and dual variables \(v\). Each variable is updated by minimizing the augmented Lagrangian \(L_\rho\) with respect to that variable while keeping the others fixed. \textbf{Right:} Each component \(b_i^{(t+1)}\), \(\beta_i^{(t+1)}\), \(v_i^{(t+1)}\) for \(i = 1, \dots, p\) is updated independently across the decision variables, which enables fast parallel updates on GPUs.}
\label{fig:parallel_comp}
\end{figure}

\subsubsection{Dual Bounds} Our ADMM-based lower bound algorithm is iterative and produces inexact primal solutions to problem \eqref{eq:ADMM1} in practice, which may not provide a valid lower bound. To ensure that our BnB framework prunes the search space correctly, we need the dual bound of problem \eqref{eq:ADMM1}. We first present a dual of problem \eqref{eq:ADMM1} and its optimality conditions in the following proposition.
\begin{proposition}\label{thm:dual}
We define the function $h:\mathbb{R}_{\ge 0}\to\mathbb{R}$ as
\begin{equation}\label{eq:hdef}
h(x)=
\begin{cases}
x^2/(4\lambda_2)-\lambda_0, & \text{if } x\le 2M\lambda_2,\\
Mx-\lambda_0-\lambda_2M^2, & \text{otherwise.}
\end{cases}
\end{equation}
A dual optimization form for problem \eqref{eq:ADMM1} is
\begin{equation}\label{eq:dual}
\max_{r\in\mathbb{R}^n}\;
-\frac12\|r\|_2^2+y^\top r-\sum_{i\in[p]}\nu_i\!\left(|X_i^\top r|\right),
\end{equation}
where
\begin{equation}\label{eq:nudef}
\nu_i(x)=
\begin{cases}
0, & \text{if } i\in \mathcal{F}_0,\\
h(x), & \text{if } i\in \mathcal{F}_1,\\
[h(x)]_+, & \text{if } i\notin \mathcal{F}_0\cup\mathcal{F}_1 \text{ and } \sqrt{\frac{\lambda_0}{\lambda_2}}\le M,\\
[Mx-\lambda_0-\lambda_2M^2]_+, & \text{if } i\notin \mathcal{F}_0\cup\mathcal{F}_1 \text{ and } \sqrt{\frac{\lambda_0}{\lambda_2}}> M.
\end{cases}
\end{equation}
Let $p^\star$ and $d^\star$ denote the optimal values of \eqref{eq:ADMM1} and \eqref{eq:dual}, respectively. Then the following hold:
\begin{enumerate}
    \item For every $r\in\mathbb{R}^n$,
    \[
    -\frac12\|r\|_2^2+y^\top r-\sum_{i\in[p]}\nu_i\!\left(|X_i^\top r|\right)\le p^\star,
    \]
which yields a valid lower bound on the optimal value of \eqref{eq:ADMM1}.

    \item Strong duality holds, i.e.
    \[
    d^\star = p^\star.
    \]

    \item If $(b^\star,\beta^\star)$ is an optimal primal solution to \eqref{eq:ADMM1}, then
    \[
    r^\star := y-Xb^\star
    \]
    is an optimal dual solution to \eqref{eq:dual}.
\end{enumerate}
\end{proposition}

Motivated by the optimality condition provided in Proposition \ref{thm:dual}, given an inexact primal solution $(\hat b,\hat \beta)$ obtained from ADMM, we set $\hat r=y-X\hat b$ as the approximate dual solution. We also note that similar to the update of $\beta$, the computation of the dual bound can exploit GPU acceleration. 

\subsubsection{ADMM Warm Starting.}
For the root node, we initialize the ADMM variables ($\beta,b$ and $v$) to zero. For child nodes, we warm-start from the parent solution, modifying the iterates to respect the newly fixed coordinates in $\mathcal{F}_0$ and $\mathcal{F}_1$ (e.g., setting $\beta_i=0$ for $i\in\mathcal{F}_0$), and then updating $b$ and $v$ via \eqref{eq:b_update} and \eqref{eq:v_i-update}. Since a child differs from its parent only by a small number of additional fixings, this typically serves as a good initialization and substantially reduces the number of ADMM iterations per node. This is a key advantage of ADMM as a first-order method in branch-and-bound. The variables $(b,\beta,v)$ transfer directly across related nodes, and the matrix $D$ is reused across the entire tree. By contrast, interior-point methods are generally less effective at exploiting warm-starts when constraints change, so they tend to require more work per node even when successive subproblems are very similar \cite{pacaud2023}.

\subsection{Parallelizing the Lower Bound Algorithm Across the BnB Tree}\label{subsect:ADMM_parallel}

In solving mixed-integer optimization problems via BnB algorithms with many binary variables, such as (\ref{eq:perspective}), a key challenge is efficiently traversing the many subproblems that arise via branching. Each node in the BnB tree represents a subproblem defined by specific variable fixations, and as the tree expands, the number of open nodes (i.e., subproblems yet to be solved) can grow exponentially. Here we discuss how we leverage the structure of our problem to parallelize computations across several open nodes in the search tree.

Consider \( K \) subproblems corresponding to different open nodes in the BnB tree. Each subproblem $k \in [K]$ is characterized by its own fixed sets \( \mathcal{F}_0^{(k)} \), \( \mathcal{F}_1^{(k)} \),  representing variables fixed to zero or one, respectively. We seek to solve these subproblems simultaneously by batching their ADMM updates. For parallel computation, we represent the variables across all subproblems using batched matrices:

\begin{itemize}
    \item \( B \in \mathbb{R}^{K \times p} \): Batch of primal variables \( \{b^{(k)}\}_{k=1}^K \)
    \item \( \boldsymbol{\beta} \in \mathbb{R}^{K \times p} \): Batch of auxiliary variables \{\( \beta^{(k)}\}_{k=1}^K \)
    \item \( V \in \mathbb{R}^{K \times p} \): Batch of dual variables \( \{v^{(k)}\}_{k=1}^K \)
\end{itemize}

Each row \( k \in [K] \) in these matrices corresponds to the variables of subproblem \( k \), allowing us to perform operations across all subproblems simultaneously, as depicted in \autoref{fig:parallel_comp}

\noindent \textbf{Update for $B$.} The primal update for all subproblems can be computed as
\begin{equation}\label{eq:batched_b_update}
    B^{(t+1)} = \big( C + \rho \boldsymbol{\beta}^{(t)} - V^{(t)} \big) D,
\end{equation}
where $D = \left( X^\top X + \rho I_p \right)^{-1} \in \mathbb{R}^{p\times p}$ is precomputed and $C := \mathbf{1}_K (X^\top y)^\top \in \mathbb{R}^{K\times p}$
is the matrix whose $k$-th row equals $(X^\top y)^\top$, with $\mathbf{1}_K \in \mathbb{R}^K$ being a $K$-dimensional vector with 1 at every entry. In other words, $C$ replicates the vector 
$X^\top y$ across the $K$ subproblems. For efficient storage, we implement a shared memory reference such that we only store the vector $X^\top y$ itself (without duplicate copies).

\noindent \textbf{Update for $\boldsymbol{\beta}$.} The update for \( \boldsymbol{\beta} \) applies the proximal operator \( T \) element-wise:
\begin{equation}\label{eq:batched_beta_update}
    \boldsymbol{\beta}^{(t+1)}_{k,i} = T\left( \tilde{\boldsymbol{\beta}}_{k,i}; a_{k,i}, M \right),
\end{equation}
where \( \tilde{\boldsymbol{\beta}} = B^{(t+1)} + {V^{(t)}}/{\rho} \), and \( a_{k,i} \) and \( m_{k,i} \) are parameters that depend on the fixations in \( \mathcal{F}_0^{(k)} \) and \( \mathcal{F}_1^{(k)} \) for each subproblem \( k \). To efficiently implement this update on GPUs, we create masks for each subproblem \( k \) and coordinate \( i \):
\[
\delta_{k,i}^{(0)} = \begin{cases}
1, & \text{if } i \in \mathcal{F}_0^{(k)}, \\
0, & \text{otherwise};
\end{cases}
\quad
\delta_{k,i}^{(1)} = \begin{cases}
1, & \text{if } i \in \mathcal{F}_1^{(k)}, \\
0, & \text{otherwise};
\end{cases}
\quad
\delta_{k,i}^{(f)} = 1 - \delta_{k,i}^{(0)} - \delta_{k,i}^{(1)}.
\]
Using these masks, we express the auxiliary variable update for all subproblems and coordinates in vectorized form as:
\begin{equation}\label{eq:beta_update_vector}
\boldsymbol{\beta}^{(t+1)} = \boldsymbol{\delta}^{(0)} \circ \mathbf{0} + \boldsymbol{\delta}^{(1)} \circ \boldsymbol{\beta}^{(t+1)}_{\text{fixed}} + \boldsymbol{\delta}^{(f)} \circ \boldsymbol{\beta}^{(t+1)}_{\text{free}},
\end{equation}
where ``\( \circ \)'' denotes element-wise multiplication; \( \boldsymbol{\beta}^{(t+1)}_{\text{fixed}} \in \mathbb{R}^{K \times p} \) is the update for variables fixed to one, computed as:
  \[
\boldsymbol{\beta}^{(t+1)}_{\text{fixed}} = \operatorname{proj}_{[-M, M]} \left( \frac{\rho}{\rho + 2 \lambda_2} \tilde{\boldsymbol{\beta}} \right);
  \]
and, \( \boldsymbol{\beta}^{(t+1)}_{\text{free}} \in \mathbb{R}^{K \times p} \) is the update for free variables, computed using operator \( T \):
  \[
\beta^{(t+1)}_{\text{free}, k,i} = \begin{cases}
T\left( \tilde{\beta}_{k,i}; \dfrac{2 \sqrt{\lambda_0 \lambda_2}}{\rho}, M \right), & \text{if } \sqrt{\dfrac{\lambda_0}{\lambda_2}} \leq M, \\
T\left( \tilde{\beta}_{k,i}; \dfrac{\lambda_0}{M \rho} + \dfrac{\lambda_2 M}{\rho}, M \right), & \text{if } \sqrt{\dfrac{\lambda_0}{\lambda_2}} > M.
\end{cases}
\]

The use of binary masks to handle variable fixations is very well-suited for GPU acceleration. Rather than explicitly maintaining different sets of active variables for each subproblem, the masks allow us to compute updates for all variables and then efficiently zero out or modify values based on their fixed status. This approach avoids conditional branches in the GPU code, which could otherwise lead to thread divergence and reduce performance.

\noindent \textbf{Update for $V$.} The dual variables \( V \) are updated simultaneously for all subproblems, given by the matix notation:
\begin{equation}\label{eq:batched_v_update}
    V^{(t+1)} = V^{(t)} + \rho \left( B^{(t+1)} - \boldsymbol{\beta}^{(t+1)} \right).
\end{equation}

This batched formulation in equations (\ref{eq:batched_b_update})-(\ref{eq:batched_v_update}) offers several key advantages. First, it allows us to leverage the high throughput of modern GPUs by processing multiple subproblems simultaneously using common instructions (see Figure \ref{fig:SIMD}). The matrix operations  can be efficiently computed using GPU-optimized linear algebra libraries. Second, all subproblems share the same data matrix $X$ at the matrix $D$, to avoid memory consumption becoming a bottleneck on GPUs. As discussed in \autoref{sec:node_parallel}, this batched formulation can process far more nodes in parallel and drastically reduces the time to close the optimality gap.

\begin{figure}[H]
\centering
\begin{tikzpicture}[
  auto matrix/.style={
    matrix of nodes,
    draw, thick, inner sep=0pt,
    nodes in empty cells, column sep=-0.2pt, row sep=-0.2pt,
    cells={
      nodes={
        minimum width=1.9em, minimum height=1.9em,
        draw, very thin, anchor=center,
        execute at begin node={%
          $\vphantom{x_|}%
          \ifnum\the\pgfmatrixcurrentrow<4
            \ifnum\the\pgfmatrixcurrentcolumn<4
              \varSymbol{#1}^{\the\pgfmatrixcurrentrow}_{\the\pgfmatrixcurrentcolumn}%
            \else 
              \ifnum\the\pgfmatrixcurrentcolumn=5
                \varSymbol{#1}^{\the\pgfmatrixcurrentrow}_{K}%
              \fi
            \fi
          \else
            \ifnum\the\pgfmatrixcurrentrow=5
              \ifnum\the\pgfmatrixcurrentcolumn<4
                \varSymbol{#1}^{T}_{\the\pgfmatrixcurrentcolumn}%
              \else
                \ifnum\the\pgfmatrixcurrentcolumn=5
                  \varSymbol{#1}^{T}_{K}%
                \fi 
              \fi
            \fi
          \fi  
          \ifnum\the\pgfmatrixcurrentrow\the\pgfmatrixcurrentcolumn=14
            \cdots
          \fi
          \ifnum\the\pgfmatrixcurrentrow\the\pgfmatrixcurrentcolumn=41
            \vdots
          \fi
          \ifnum\the\pgfmatrixcurrentrow\the\pgfmatrixcurrentcolumn=44
            \ddots
          \fi$
        }
      }
    }
  }
]

\newcommand{\varSymbol}[1]{%
  \ifthenelse{\equal{#1}{beta}}{\beta}{#1}%
}

\matrix[auto matrix=v, xshift=3em, yshift=3em, cells={nodes={fill=orange!25}}] (matv) {
  & & & & \\
  & & & & \\
  & & & & \\
  & & & & \\
  & & & & \\
};

\matrix[auto matrix=beta, xshift=1.5em, yshift=1.5em, cells={nodes={fill=blue!25}}] (matbeta) {
  & & & & \\
  & & & & \\
  & & & & \\
  & & & & \\
  & & & & \\
};

\matrix[auto matrix=b, cells={nodes={fill=red!25}}] (matb) {
  & & & & \\
  & & & & \\
  & & & & \\
  & & & & \\
  & & & & \\
};

\draw[thick, -stealth] ([xshift=1ex]matb.south east) -- ([xshift=1ex]matv.south east)
  node[midway, below, xshift=1.5em] {Variables};
\draw[thick, -stealth] ([yshift=-1ex]matb.south west) -- 
  ([yshift=-1ex]matb.south east) node[midway, below] {Subproblems};
\draw[thick, -stealth] ([xshift=-1ex]matb.north west)
  -- ([xshift=-1ex]matb.south west) node[midway, above, rotate=90] { Iterations};
\end{tikzpicture}
\caption{SIMD model for parallel lower bound algorithm}
\label{fig:SIMD}
\end{figure}
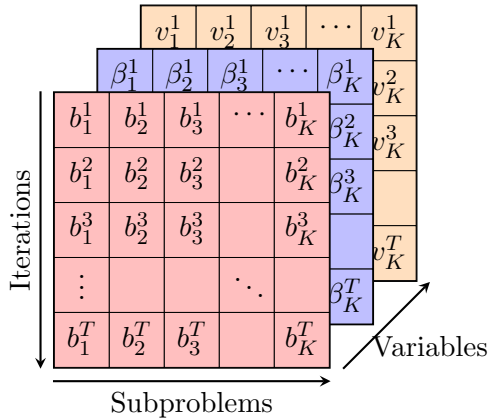

Our shared memory reference scheme exploits a key observation: in the branch-and-bound tree, all subproblems operate on the same underlying data matrices ($X$, $D$ and $y$), with differences only in the variables $b, \beta$, and $v$. Instead of naively replicating these matrices for each subproblem, which would quickly exhaust GPU memory, we maintain a single global reference to these common data. Each subproblem stores only its unique decision variables.

\subsection{Upper bound algorithm}\label{sec: upper_bound}
During the BnB algorithm execution, whenever we find an integral $z$-solution, it provides an upper bound for Problem~\eqref{eq:perspective}. Upper bounds (aka primal solutions) are critical in BnB algorithms for efficiently pruning the search space and avoiding exhaustive enumeration. We discuss how to efficiently compute a good solution for problem \eqref{eq:perspective} within our BnB framework. Given a solution $(\hat \beta,\hat z,\hat s)$ of the relaxation problem \eqref{eq:relaxnode}, we round each component of $\hat z$ to get an integral solution $z^*$. With fixed $z=z^*$, problem \eqref{eq:perspective} reduces to a quadratic problem with box constraint:
\begingroup
\setlength{\abovedisplayskip}{4pt}
\setlength{\belowdisplayskip}{4pt}
\setlength{\abovedisplayshortskip}{2pt}
\setlength{\belowdisplayshortskip}{2pt}
\begin{equation}\label{eq:upperboundbeta}
\begin{aligned}
\min_{\beta_\mathcal{S}} \quad & U_\mathcal{S}(\beta_\mathcal{S})
:=\frac{1}{2}\|y-X_\mathcal{S}\beta_\mathcal{S}\|_2^2+\lambda_2\|\beta_\mathcal{S}\|_2^2 \quad 
\text{s.t.} \quad |\beta_i|\le M,\; i\in\mathcal{S}
\end{aligned}
\end{equation}
\endgroup
where $X_\mathcal{S}$ denotes the submatrix of $X$ with columns in $\mathcal{S}=\{i\mid z_i=1\}$. We solve \eqref{eq:upperboundbeta} using a fast proximal gradient method with Nesterov acceleration and backtracking line search~\cite{beck2009fast}. At iteration $t$, given the current and previous iterates $\beta_{\mathcal{S}}^{(t)}$ and $\beta_{\mathcal{S}}^{(t-1)}$, we first form an extrapolated point
\begingroup
\setlength{\abovedisplayskip}{3pt}
\setlength{\belowdisplayskip}{3pt}
\setlength{\abovedisplayshortskip}{2pt}
\setlength{\belowdisplayshortskip}{2pt}
\begin{equation}\label{eq:fpg_extrapolate}
\begin{aligned}[t]
\widetilde{\beta}_{\mathcal{S}}^{(t)}
&=
\beta_{\mathcal{S}}^{(t)}
+ \frac{t}{t+3}\bigl(\beta_{\mathcal{S}}^{(t)} - \beta_{\mathcal{S}}^{(t-1)}\bigr),
\end{aligned}
\end{equation}
\endgroup
then compute the gradient of the smooth part:
\begingroup
\setlength{\abovedisplayskip}{3pt}
\setlength{\belowdisplayskip}{3pt}
\setlength{\abovedisplayshortskip}{2pt}
\setlength{\belowdisplayshortskip}{2pt}
\begin{equation}\label{eq:fpg_grad}
\begin{aligned}[t]
\nabla g\bigl(\widetilde{\beta}_{\mathcal{S}}^{(t)}\bigr)
&=
X_{\mathcal{S}}^\top\bigl(X_{\mathcal{S}}\,\widetilde{\beta}_{\mathcal{S}}^{(t)} - y\bigr)
+ 2\lambda_2\,\widetilde{\beta}_{\mathcal{S}}^{(t)},
\end{aligned}
\end{equation}
\endgroup
and form a projected gradient step with stepsize $\alpha^{(t)}$:
\begingroup
\setlength{\abovedisplayskip}{3pt}
\setlength{\belowdisplayskip}{3pt}
\setlength{\abovedisplayshortskip}{2pt}
\setlength{\belowdisplayshortskip}{2pt}
\begin{equation}\label{eq:fpg_step}
\begin{aligned}[t]
\beta_{\mathcal{S}}^{(t+1)}
&=
\Pi_{[-M,M]}\Bigl(
\widetilde{\beta}_{\mathcal{S}}^{(t)}
- \alpha^{(t)} \nabla g\bigl(\widetilde{\beta}_{\mathcal{S}}^{(t)}\bigr)
\Bigr),
\end{aligned}
\end{equation}
\endgroup
where $\Pi_{[-M,M]}$ denotes the elementwise projection onto $[-M,M]$. The stepsize $\alpha^{(t)}$ is chosen via a standard backtracking line search: we shrink $\alpha^{(t)}$ until the Armijo condition:
\begingroup
\setlength{\abovedisplayskip}{4pt}
\setlength{\belowdisplayskip}{4pt}
\setlength{\abovedisplayshortskip}{2pt}
\setlength{\belowdisplayshortskip}{2pt}
\begin{equation}\label{eq:fpg_armijo}
\begin{aligned}[t]
U_\mathcal{S}\bigl(\beta_{\mathcal{S}}^{(t+1)}\bigr)
&\le
U_\mathcal{S}\bigl(\widetilde{\beta}_{\mathcal{S}}^{(t)}\bigr)
+ \bigl\langle \nabla g\bigl(\widetilde{\beta}_{\mathcal{S}}^{(t)}\bigr),
\beta_{\mathcal{S}}^{(t+1)} - \widetilde{\beta}_{\mathcal{S}}^{(t)} \bigr\rangle
+ \frac{1}{2\alpha^{(t)}} \bigl\|\beta_{\mathcal{S}}^{(t+1)} - \widetilde{\beta}_{\mathcal{S}}^{(t)}\bigr\|_2^2
\end{aligned}
\end{equation}
\endgroup
is satisfied. The procedure is terminated when the decrease in $U_\mathcal{S}$ or the change in $\beta_{\mathcal{S}}$ falls below a prescribed tolerance, at which point the resulting $\beta_{\mathcal{S}}$ provides a feasible upper bound for the original problem.

\subsection{Parallelizing the upper bound algorithm across the BnB tree}\label{sec:upper_bound_parallel}

To efficiently handle many candidate supports in parallel, we implement a vectorized version of the fast proximal gradient scheme in \eqref{eq:fpg_extrapolate}--\eqref{eq:fpg_armijo} that, similar to the algorithm presented in \autoref{subsect:ADMM_parallel}, can handle $K$ problems in parallel. Instead of solving \eqref{eq:upperboundbeta} separately for each support, we stack all coefficient vectors into a single matrix $\mathbf{B} \in \mathbb{R}^{K \times p}$, where row $k$ corresponds to subproblem $k$, and encode the supports via a binary mask $\mathbf{M} \in \{0,1\}^{K \times p}$, with $\mathbf{M}_{k,j} = 1$ if $j \in \mathcal{S}^k$ and $0$ otherwise.

Each iteration applies the extrapolation step \eqref{eq:fpg_extrapolate} in a batched fashion:
\[
\widetilde{\mathbf{B}}^{(t)} 
= \mathbf{B}^{(t)} 
+ \frac{t}{t+3}\big(\mathbf{B}^{(t)} - \mathbf{B}^{(t-1)}\big),
\]
which is simply \eqref{eq:fpg_extrapolate} applied row-wise to $\mathbf{B}^{(t)}$. We then compute all gradients simultaneously, in analogy with \eqref{eq:fpg_grad}, via:
\[
\nabla G\big(\widetilde{\mathbf{B}}^{(t)}\big)
= \big(\mathbf{X}^\top(\mathbf{X} \,\widetilde{\mathbf{B}}^{(t)\top} - \mathbf{y}\mathbf{1}^\top)\big)^\top
+ 2\lambda_2 \,\widetilde{\mathbf{B}}^{(t)},
\]
and enforce the support constraints by elementwise masking $\nabla G \leftarrow \nabla G \odot \mathbf{M}$. The projected gradient step \eqref{eq:fpg_step} is likewise applied in one shot to all rows:
\[
\mathbf{B}^{(t+1)} 
= \Pi_{[-M,M]}\big(\widetilde{\mathbf{B}}^{(t)} - \boldsymbol{\alpha} \odot \nabla G(\widetilde{\mathbf{B}}^{(t)})\big) \odot \mathbf{M},
\]
where $\boldsymbol{\alpha} \in \mathbb{R}^K$ stores per-subproblem stepsizes and $\odot$ denotes broadcasting and elementwise multiplication over rows. Finally, the backtracking line search enforces a batched version of the Armijo condition \eqref{eq:fpg_armijo}: we maintain a boolean “active” mask over subproblems, update only the active rows of $\mathbf{B}$ and $\boldsymbol{\alpha}$, and evaluate all corresponding trial losses via batched matrix–matrix products until the inequality in \eqref{eq:fpg_armijo} holds for each active row. These updates are therefore direct vectorized analogues of \eqref{eq:fpg_extrapolate}--\eqref{eq:fpg_armijo}, giving a highly parallel upper-bound algorithm that can process large batches of supports from the BnB tree in a single pass.

\subsection{Initializing the BnB tree with specialized heuristic}\label{sec:MP_outline}

The upper bound algorithm above (Sec~\ref{sec: upper_bound}) is used to quickly obtain a feasible solution given a pre-specified support. 
Thus we also design a heuristic algorithm to generate a high-quality support prior to running the BnB algorithm, starting from scratch. This provides a high-quality upper bound prior to initializing the BnB algorithm. The algorithm is inspired by the classical matching pursuit procedure, a lightweight method based on greedy selection \cite[see also references therein]{freund2015newperspectiveboostinglinear} with forward and backward phases. We adapt the method to use GPU acceleration (especially in selecting the coordinate to update)---see Sec~\ref{sec:matching_pursuit} for details.

\subsection{Considerations and Challenges for BnB on GPUs}\label{sec:challenges}

While GPUs have great promise in BnB, they also bring unique challenges and considerations. We outline below some challenges we have encountered with integrating BnB on GPUs and how we address them. 

\noindent \textbf{Modes of parallelism.} GPUs are highly optimized for SIMD operations, where the same or similar instructions are executed concurrently on multiple threads. If the operations are highly dissimilar this can lead to thread divergence, i.e. a situation in which the same group of threads take different execution paths which forces serialization. The choice of subproblem algorithm itself determines the degree of parallelism that can be achieved. While algorithms like coordinate descent may be more efficient for CPU execution due to their sequential nature, we employ ADMM as it enables decomposition of the problem into subproblems with parallel coordinates. Moreover, we carefully chose the splitting and updates for maximum separability in the coordinate subproblems.

\noindent  \textbf{GPU-CPU communication.} Transferring large data structures regularly from the CPU to the GPU can introduce communication overhead and slow down the execution of BnB. We have designed our algorithm such that the lightweight serial logic (branching, bounding, pruning, etc) is performed on CPUs while the subproblem algorithms are executed on GPUs (see 
Algorithm \ref{algo:BnB-approx}). The only data that needs to be passed between the CPU and GPU are the bound evaluations and, when warm starting is used, the decision variables for a given node.

\noindent \textbf{Memory.} Modern GPUs such as NVIDIA A100, H100, and V100 typically have 32-96GBs of RAM. If not managed properly, this memory can be exhausted very quickly for large instances. We address this by implementing shared memory references for the vectors and matrices required for the lower-bound and upper-bound algorithms in 
Section~\ref{sect:gpubnb}. This means the same vectors and matrices ($D$, $X$, $y$) can be reused at each subproblem.

\section{Experiments}
We compare our BnB algorithm against state-of-the-art approaches for sparse learning in high-dimensional settings. We consider three aspects: computational efficiency compared to existing methods, sensitivity to parameter settings, and performance across diverse datasets. For 
Section~\ref{sec:expts-comp1}, we turn off node parallelism to compare with existing BnB methods which also process nodes sequentially. In Section~\ref{sec:node_parallel}, we show that node parallelism further accelerates this algorithm both in terms of the number of nodes processed and the MIP optimality gap.
\vspace{-3mm}
\subsection{Experimental setup} \label{sect:expsetting}

Our algorithm, GPUBnB, is implemented in Python using {PyTorch} \cite{paszke2017automatic} for GPU acceleration via NVIDIA CUDA. We evaluate it on an NVIDIA A100-PCIe-80GB GPU against four baselines:~(i)~the CPU version of our algorithm (CPUBnB),~(ii)~L0BnB \cite{hazimeh2021sparse} implemented in Python with {Numba} optimization \cite{lam2015numba},and~(iii)~MOSEK \cite{mosek} and GUROBI \cite{gurobi}. All CPU-based methods run on 20 AMD EPYC 9474F cores (3.00 GHz, 80GB RAM, 4 threads).


\noindent \textbf{Datasets.} For synthetic datasets, following \cite{hazimeh2020fast,hazimeh2021sparse}, we generate the data matrix $X \in \mathbb{R}^{n \times p}$ by drawing each row from a multivariate Gaussian distribution $\mathcal{N}(0, \Sigma)$, where $\Sigma = \rho \mathbf{1}_{p\times p} + (1-\rho) \mathbf{I}_{p}$ controls feature correlation through parameter $\rho$. The ground truth coefficient $\beta^{\dagger} \in \mathbb{R}^p$ is constructed with $k^0$ equispaced nonzero entries of value $1$, with remaining entries set to $0$. We generate the response vector as $y = X\beta^{\dagger} + \varepsilon$, where $\varepsilon \sim \mathcal{N}(0, \sigma^2I_n)$ represents i.i.d. Gaussian noise. The signal-to-noise ratio, defined as $\mathrm{SNR} = \operatorname{Var}(X\beta^{\dagger})/\sigma^2$, is used to control problem difficulty: lower SNR indicates higher noise levels relative to the signal. Unless otherwise specified, we set $\rho = 0.2$, $k^0 = 10$. We adjust SNR based on sample size $n$: $(\mathrm{SNR}, n) = (10, 1000)$, $(10/3, 3000)$, and $(0.5, 10000)$.

We also evaluate our proposed algorithm on the real-world UJIIndoorLoc dataset \cite{ujiindoorloc_310}, which focuses on indoor positioning using WLAN/WiFi fingerprinting across multiple buildings and floors. The dataset is split into training (60\%), validation (20\%), and test (20\%) sets. To create a higher-dimensional problem, we augment the feature space by replicating each feature 49 times with random permutations, resulting in a dataset with $n=11,962$ samples and $p=23,250$ features. We preprocess the data by mean-centering and normalizing both the response variable and the feature matrix columns.

\noindent \textbf{Selection of $\lambda_0$, $\lambda_2$, and $M$.} Following \cite{hazimeh2021sparse}, for synthetic datasets we first determine $\lambda_2$ by analyzing the ridge regression solution restricted to the true support $S^{\dagger} = \operatorname{supp}(\tilde{\beta}^{\dagger})$. Specifically, for fixed $\lambda_2$, we define $\beta(\lambda_2)$ as:
\begin{equation}
   \beta(\lambda_2) \in \underset{\beta \in \mathbb{R}^p}{\arg \min} \frac{1}{2}\|y-X\beta\|_2^2 + \lambda_2\|\beta\|_2^2 \quad \text{s.t.} \quad \beta_{(S^{\dagger})^c} = 0.
\end{equation}
The optimal $\lambda_2^*$ is then determined by minimizing the $\ell_2$ estimation error:
\begin{equation}
   \lambda_2^* \in \underset{\lambda_2 \geq 0}{\arg \min} \|\tilde{\beta}^{\dagger}-\beta(\lambda_2)\|_2.
\end{equation}
We estimate $\lambda_2^*$ through grid search over $[10^{-4}, 10^4]$ using 100 logarithmically-spaced points. In our experiments, we typically set $\lambda_2 = \lambda_2^*$, though we also examine other values as fractions or multiples of $\lambda_2^*$. After determining $\lambda_2$, we follow the adaptive selection rules in \cite{hazimeh2020fast} to compute a regularization path of decreasing $\lambda_0$ values: $\lambda_0^1(\lambda_2) > \lambda_0^2(\lambda_2) > \cdots > \lambda_0^m(\lambda_2)$. For each $\lambda_0^i(\lambda_2)$, we apply GPUBnB~to solve problem \eqref{eq:perspective} with $\lambda_0=\lambda_0^i(\lambda_2), \lambda_2=\lambda_2$, $M=1.05\|\beta(\lambda_2)\|_{\infty}$ and a 120 seconds time limit to obtain approximate solution $\hat{\beta}(\lambda_0^i(\lambda_2))$. We select  $\lambda_0^*(\lambda_2)$ that maximizes the F-measure between the support of the approximate solution, i.e., $\operatorname{supp}(\hat{\beta}(\lambda_0^i(\lambda_2),\lambda_2))$, and the true support $S^{\dagger}$. For the big-M parameter, we set $M^*(\lambda_2) = \|\hat \beta(\lambda_2)\|_{\infty}$. Unless otherwise specified, our experiments use $\lambda_0 = \lambda_0^*(\lambda_2)$ and $M = 1.5M^*(\lambda_2^*)$.

For real datasets, where the true coefficient vector $\tilde{\beta}^{\dagger}$ is unknown, we use a different procedure. We compute a grid over $\lambda_0$ and $\lambda_2$: We select 20 log-spaced values of $\lambda_2$ in $[10^{-3}, 10^3]$ and vary $\lambda_0$ to obtain solutions with different sparsity levels. For each parameter pair $(\lambda_0,\lambda_2)$, we apply GPUBnB to solve problem \eqref{eq:perspective} with a 120-second time limit, obtaining approximate solution $\hat{\beta}(\lambda_0,\lambda_2)$. The optimal regularization parameters are then selected using the validation set to minimize least squares loss. For the big-M parameter, similarly, we set $M = 1.5\|\hat \beta(\lambda_0,\lambda_2)\|_{\infty}$.

\noindent \textbf{BnB Settings}.
All BnB methods are applied to the perspective formulation~\eqref{eq:perspective}. For each BnB algorithm, we define the relative optimality gap between the upper bound $UB$ and lower bound $LB$ as $(UB - LB)/UB$. We terminate the algorithm when the relative optimality gap falls below 1\% or when reaching either the maximum runtime of 4 hours or the memory limit of 80GB. Additionally, at the node level, we terminate each relaxation subproblem in the BnB tree when its primal-dual gap becomes smaller than $10^{-4}$.

\vspace{-3mm}
\subsection{Comparison with state-of-the-art BnB methods on synthetic datasets}\label{sec:expts-comp1}

\subsubsection{Varying number of samples and features}
We examine BnB algorithm performance across different problem dimensions. We generate synthetic datasets with varying sizes: the number of samples $n \in \{10^3,3\cdot 10^3,10^4\}$ and the number of features $p \in \{10^4,3\cdot10^4,10^5\}$, fixing $k^0=10$ nonzeros. Parameters $\lambda_0$, $\lambda_2$, and $M$ are selected as described in Section \ref{sect:expsetting}. Table~\ref{tab:np} presents average runtimes over 10 random seeds for each method. We keep node parallelism off for comparison against other methods that do not use node parallelism. The results reveal several key findings. First, compared to the commercial solver MOSEK, the other three methods achieve more than 10x speedup for the smallest problem size ($n=10^3$, $p=10^4$) and can handle larger instances where MOSEK fails due to memory constraints. When we compare among CPU-based methods, L0BnB~\cite{hazimeh2021sparse} generally outperforms CPUBnB, especially when the sample-to-feature ratio ($n/p$) is small, while CPUBnB~shows slightly better performance for larger $n/p$ ratios. We believe this is due to coordinate descent being a superior node relaxation algorithm for small $n$, large $p$ instances (see discussion below). 
Finally, our GPU implementation (GPUBnB) achieves strong speedups of 5-35x over its CPU counterpart (CPUBnB)---the acceleration is more pronounced as the problem size $n \times p$ increases.


\begin{table}[htbp]
\centering
\renewcommand{\arraystretch}{0.6}
\begin{tabular}{l|l|ccccc} \toprule[0.7pt]
\small
$n$ & $p$  & GPUBnB & CPUBnB & L0BnB & MOSEK & GUROBI  \\ \midrule
\multirow{3}{*}{$10^3$}& $10^4$& 23.7 & 149 & 46.9 & 1796 & -\\ 
& $3\!\cdot\!10^4$& 58.4 & 967 & 225 & - & -\\ 
& $10^5$& 355 & 9182 & 847  & - & -\\  \hline
\multirow{3}{*}{$3\!\cdot\!10^3$}& $10^4$& 25.7 & 358 & 74.8 & -& - \\ 
& $3\!\cdot\!10^4$& 67.1 & 1728 & 177 & - & -\\ 
& $10^5$& 272 & 8958 & 504 & - & -\\  \hline
\multirow{3}{*}{$10^4$}& $10^4$& 9.91 & 279 & 328 & - & -\\ 
& $3\!\cdot\!10^4$& 79.9 & 2623  & 688 & - & -\\ 
& $10^5$& 257 & 8911 & 1514 & - & -\\ 
\bottomrule[0.7pt]
\end{tabular}
\caption{
Runtime (in seconds) of our approach versus competing methods averaged over 10 random seeds. `-' indicates an exceeded memory limit (80GB).}
\label{tab:np}
\end{table}

We further analyze the performance differences between L0BnB and our approach. Implementation matters aside, both methods employ identical branch-and-bound frameworks and solve the same MIP relaxation at each node. In both approaches, over 98\% of the runtime is devoted to solving the node relaxation problem \eqref{eq:relaxnode}, with performance differences stemming from their distinct methods: L0BnB uses coordinate descent (CD), while our approach employs ADMM. L0BnB's CD implementation, as described in \cite{hazimeh2021sparse}, incorporates several problem-specific tricks (e.g., active set updates and gradient screening) that enhance its practical performance. However, the inherently sequential nature of CD makes GPU acceleration challenging. In contrast, our ADMM-based approach readily exploits GPU parallelization. Figure \ref{fig:np} compares the runtime to solve the node relaxations for different methods across 10 random seeds. While L0BnB generally outperforms CPUBnB (in runtime) to solve the node relaxation, our GPU implementation (GPUBnB) achieves a 2-10x speedup compared to L0BnB. 
\begin{figure}
    \centering
    \includegraphics[width=0.8\linewidth]{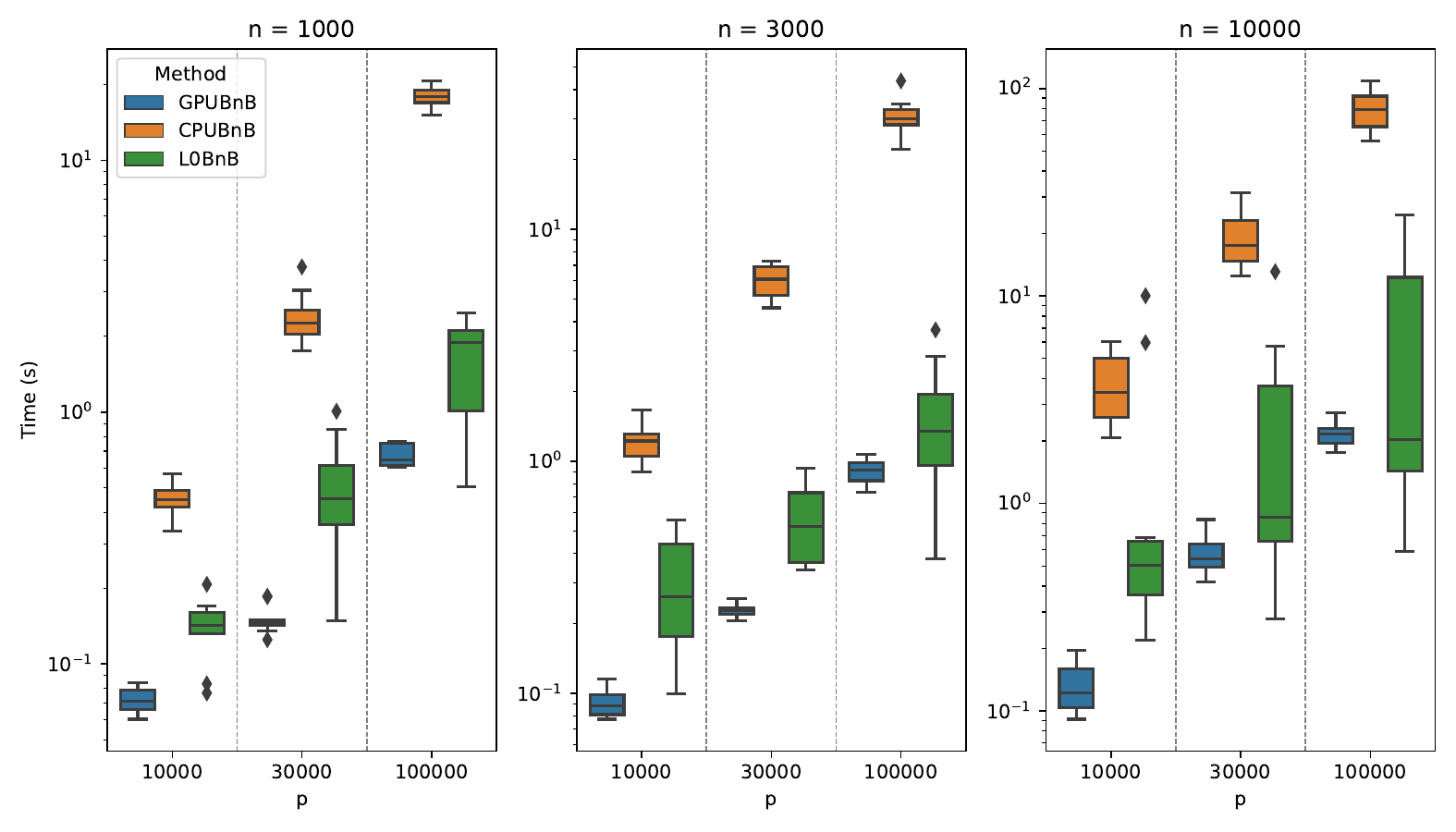}
    \caption{Average node relaxation solving time (in seconds) of our approach versus L0BnB on 10 random seeds.}
    \label{fig:np}
\end{figure}

\vspace{-3mm}
\subsubsection{Impact of problem parameters and computing platforms}

We first examine how various problem parameters affect the performance of each BnB method. We generate synthetic datasets with $n=3000$ samples and $p=30000$ features. We set baseline parameters $\mathrm{SNR}=10/3$, $k^0=10$, with $\lambda_2$ and $\lambda_0$ chosen as described in Section \ref{sect:expsetting}. We then vary one parameter at a time. Specifically, we consider $\mathrm{SNR}\in\{0.5,1,2,3,10\}$ and $k^0\in\{5,10,15,20,25\}$. For the regularization parameters, we sweep $\lambda_2 \in \{0.03, 0.1, 0.3, 3, 10, 30\}\cdot \lambda_2^*$ and $\lambda_0 \in \{0.05, 0.2, 0.4, 0.6, 1.5, 2.0\}\cdot \lambda_0^*$, where $\lambda_2^*$ and $\lambda_0^*$ are optimal parameter values, as defined in Section \ref{sect:expsetting}.

\begin{table}[htbp]
\centering
 \begin{minipage}{0.43\textwidth}\renewcommand{\arraystretch}{0.6}
        \centering
        \begin{adjustbox}{width=0.93\columnwidth,center}
        \begin{tabular}{l|ccc}
           \toprule[0.7pt]
        $\mathrm{SNR}$ & GPUBnB & CPUBnB  & L0BnB \\ \midrule
0.5& (100362) & (7258) & (1846) \\ 
1& 217 & 4992 & (1777) \\ 
2& 59.1 & 1983 & 567 \\ 
3& 56.8 & 1884 & 175\\ 
            \bottomrule[0.7pt]
        \end{tabular}
        \end{adjustbox}
    \end{minipage}
    \begin{minipage}{0.43\textwidth}\renewcommand{\arraystretch}{0.55}
        \centering
        \begin{adjustbox}{width=0.93\columnwidth,center}
        \begin{tabular}{l|ccc}
           \toprule[0.7pt]
$k^0$ & GPUBnB & CPUBnB  & L0BnB \\ \midrule
5& 19.2 & 492 & 1.38 \\ 
10& 102 & 2461 & 233 \\ 
15& 364 & 8669 & 3017 \\ 
20& 1481 & (5578) & (9597) \\ 
\bottomrule[0.7pt]
\end{tabular}
\end{adjustbox}
\end{minipage}
\caption{Runtime (in seconds) for synthetic datasets with $n=3000$ samples and $p=30000$ features, varying SNR (\textbf{left}) and number of nonzeros in the true coefficient $\beta^{\dagger}$ (\textbf{right}).  If an algorithm cannot achieve $10^{-2}$ optimality gap within 4 hours, the number of nodes solved in 4 hours is displayed in brackets.
}
\label{tab:snrk}
\end{table}

\begin{table}[htbp]
\centering
 \begin{minipage}{0.44\textwidth}\renewcommand{\arraystretch}{0.6}
        \centering
        \begin{adjustbox}{width=0.93\columnwidth,center}
        \begin{tabular}{l|ccc}
           \toprule[0.7pt]
$\lambda_2$ & GPUBnB & CPUBnB  & L0BnB \\ \midrule
0.1$\lambda_2^*$& 84.8 & 1947 & 470 \\ 
0.3$\lambda_2^*$& 78.1 & 2077 & 329 \\ 
3$\lambda_2^*$& 24.4 & 608 & 3.43 \\ 
10$\lambda_2^*$& 9.55 & 370 & 1.12 \\ 
30$\lambda_2^*$& 2.03 & 76.7 & 1.01 \\ 
            \bottomrule[0.7pt]
        \end{tabular}
        \end{adjustbox}
    \end{minipage}
    \hfill
    \begin{minipage}{0.50\textwidth}\renewcommand{\arraystretch}{0.6}
        \centering
        \begin{adjustbox}{width=0.93\columnwidth,center}
\begin{tabular}{l|ccc|c}
\toprule[0.7pt]
$\lambda_0$ & GPUBnB & CPUBnB  & L0BnB & $\|\beta\|_0$\\ \midrule
0.05$\lambda_0^*$& (40913) & (1737) & (441) & 16 \\ 
0.2$\lambda_0^*$& 1648 & (4365) & (1521) & 10 \\ 
0.4$\lambda_0^*$& 7.80 & 275 & 37.2 & 10 \\ 
0.6$\lambda_0^*$& 12.3 & 440 & 26.2 & 10 \\ 
2$\lambda_0^*$& 243 & 6399 & 21.1 & 7 \\
\bottomrule[0.7pt]
\end{tabular}
\end{adjustbox}
\end{minipage}
\caption{ Runtime (in seconds) for synthetic datasets with $n=3000$ samples and $p=30000$ features, varying $\lambda_2$ (\textbf{left}) and $\lambda_0$ (\textbf{right}). For varying $\lambda_0$, we also report the number of nonzeros in the solution $\beta$ obtained by each BnB method. If an algorithm cannot achieve $10^{-2}$ optimality gap within 4 hours, the number of nodes solved in 4 hours is displayed in brackets.
}
\label{tab:l2l0}
\end{table}
Tables \ref{tab:snrk} and \ref{tab:l2l0} compare L0BnB with our approach across these settings (we exclude MOSEK due to memory constraints). For relatively easy-to-solve problems (low noise with $\mathrm{SNR} \geq 3$ or small $k^0$), L0BnB demonstrates superior performance. However, our approach performs better for harder instances (high noise or large $k^0$), with GPUBnB achieving up to 10x speedup over L0BnB. Similar patterns emerge when considering different regularization parameters: L0BnB excels with larger $\lambda_2$ or $\lambda_0$ but slows down for smaller ones, while CPUBnB maintains stable performance and outperforms L0BnB for smaller $\lambda_2$ and  $\lambda_0$ values. Our GPU implementation (GPUBnB) consistently achieves 20-30x speedup over CPUBnB across all parameter settings. To further analyze GPU acceleration benefits, we evaluate our approach on various computing platforms (with specifications displayed in Table \ref{tab:gpu}). The results demonstrate that GPU acceleration effectiveness varies with problem size and the GPU hardware used. 
\begin{table}
    \centering
    \renewcommand{\arraystretch}{0.95}
        \begin{adjustbox}{width=0.6\columnwidth,center}
        \begin{tabular}{c|ccc}
            \toprule[0.6pt]
            platform & processor  & theoretical peak (FP64) & bandwidth \\ \midrule
            CPU &  AMD EPYC 9474F 3.60GHz & 66 GFLOPs & 460.8GB/s  \\
            V100 & NVIDIA V100-PCIe-32GB & 7.1 TFLOPs& 0.9TB/s\\ 
            A100 & NVIDIA A100-PCIe-80GB & 9.7 TFLOPs & 1.9TB/s\\ 
            L40S & NVIDIA L40S-PCIe-48GB & 1.4 TFLOPs & 0.9TB/s \\ 
            H100 & NVIDIA H100-PCIe-80GB & 25.6 TFLOPs  & 2TB/s \\ 
        \bottomrule[0.7pt]
        \end{tabular}
        \end{adjustbox}
        \caption{Specifications of the computing platforms used.}
        \label{tab:gpu}
\end{table}

\subsection{Real-world dataset}\label{sec:exp-real-world}

We evaluate GPUBnB and CPUBnB on the UJIIndoorLoc dataset ($n=11,962$ samples, $p=23,250$ features), again without node parallelism. We generate solutions with varying sparsity levels by adjusting $\lambda_0$ as multiples of $\lambda_0^*$. Figure~\ref{fig:real} displays the best optimality gaps achieved within 4 hours and the average node-solving times, respectively. On this dataset, GPUBnB processes each node approximately 30 times faster than its CPU counterpart, enabling it to achieve much smaller optimality gaps within the 4-hour time limit.

\begin{figure}[H]
    \centering
    \begin{subfigure}{0.4\textwidth} 
        \centering
        \includegraphics[width=\textwidth]{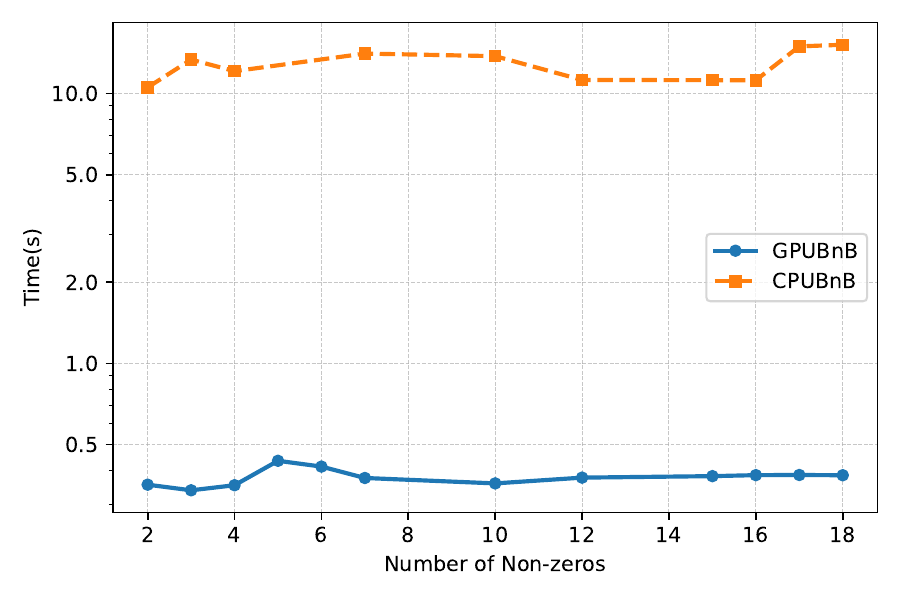} 
    \end{subfigure}
    \begin{subfigure}{0.4\textwidth} 
        \centering
        \includegraphics[width=\textwidth]{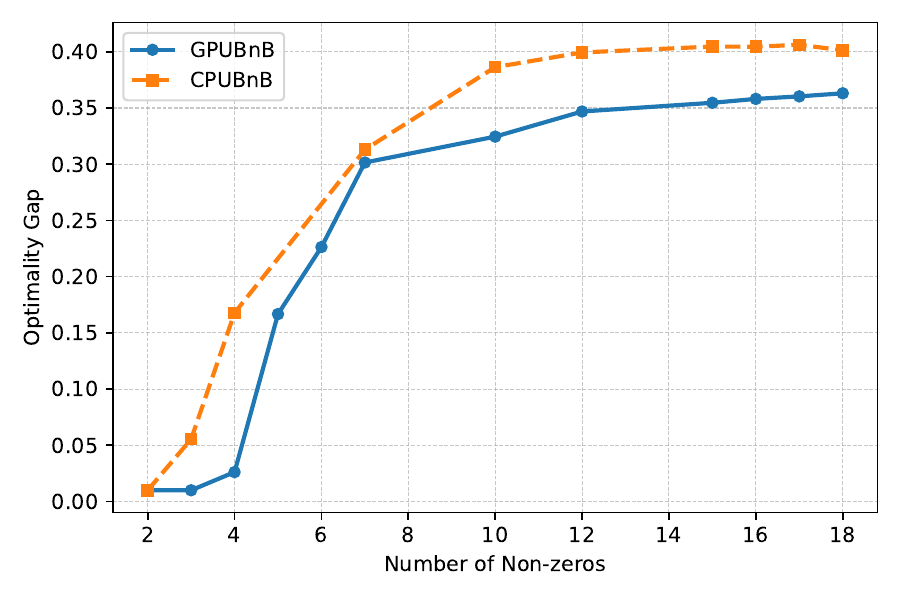} 
    \end{subfigure}
    \caption{Performance comparison on UJIIndoorLoc dataset: (\textbf{left}) average node solving time in seconds, and (\textbf{right}) best optimality gap achieved within 4 hours, for GPU and CPU implementations of our approach.}
    
    \label{fig:real}
    
\end{figure}

\subsection{Node parallelism}\label{sec:node_parallel}

As outlined in Section~\ref{subsect:ADMM_parallel}, our proposal leverages GPU parallelization by solving multiple branch-and-bound nodes simultaneously. To quantify the benefits of this approach, we conducted experiments varying the batch size $K$—defined as the number of nodes processed in parallel during the lower-bound evaluation step.

\noindent \textbf{Simulated Experiment.} We use the same data generation protocol described in Section \ref{sect:expsetting} to ensure consistency; specifically, we utilize the same synthetic setup with $n=3000$, $\rho=0.2$, $k^0 = 10$, and fix the algorithm to {GPUBnB}. However, unlike the previous experiments (in Sections~\ref{sec:expts-comp1}, \ref{sec:exp-real-world}) where the parallelization strategy was fixed to sequential ($K = 1$), we now explicitly vary the batch size $K$—the number of nodes processed simultaneously by the GPU. We test varying the feature dimensions $p \in \{10^4, 3\times10^{4}, 10^5\}$. We test a range of batch sizes $K \in \{1, 10, 50, 100\}$. Note that $K=1$ corresponds to a standard strategy without node parallelism, while e.g. $K=50$ represents a strategy of processing 50 nodes in parallel at each step. All experiments were conducted with a fixed time limit of 3600 seconds. For $p \in \{10^4, 3 \times 10^4\}$, we run the experiments on an NVIDIA A100-PCIe-80GB GPU and for $p = 10^5$ we conduct the experiment on a NVIDIA H100-PCIe-80GB GPU.

\noindent \textbf{Results.} Figure \ref{fig:node_parallel_syn} demonstrates the impact of different batch sizes on {GPUBnB} performance for the synthetic instances. The top plot illustrates the MIP optimality gap reduction over time, while the bottom plot tracks the cumulative number of nodes processed.
We observe that increasing the batch size can provide significant improvements in node throughput. For the largest problem setting $p=10^5$, using a batch size of $K=100$ allows the algorithm to process over 25x as many nodes per second compared to the sequential baseline. This throughput advantage translates directly to convergence speed. Larger batch sizes reduce the optimality gap down much faster in the early stages of the search. While larger batches can theoretically introduce some search overhead (processing nodes that might have been pruned given updated bounds from a sequential search), the speedup provided by saturating the GPU cores appear to greatly outweigh this inefficiency (at least in the examples we consider). Crucially, this strong increase in throughput does not lead to a proportional increase in GPU memory consumption. For example, at $p=3 \times 10^4$, increasing the batch size from $K=1$ to $K=100$ results in a 32\% increase in peak memory, despite a massive gain in processing speed. This efficiency is achieved through our shared memory reference scheme described in section \ref{sec:challenges}.

\begin{figure}[H]
    \centering
    \includegraphics[width=0.75\linewidth]{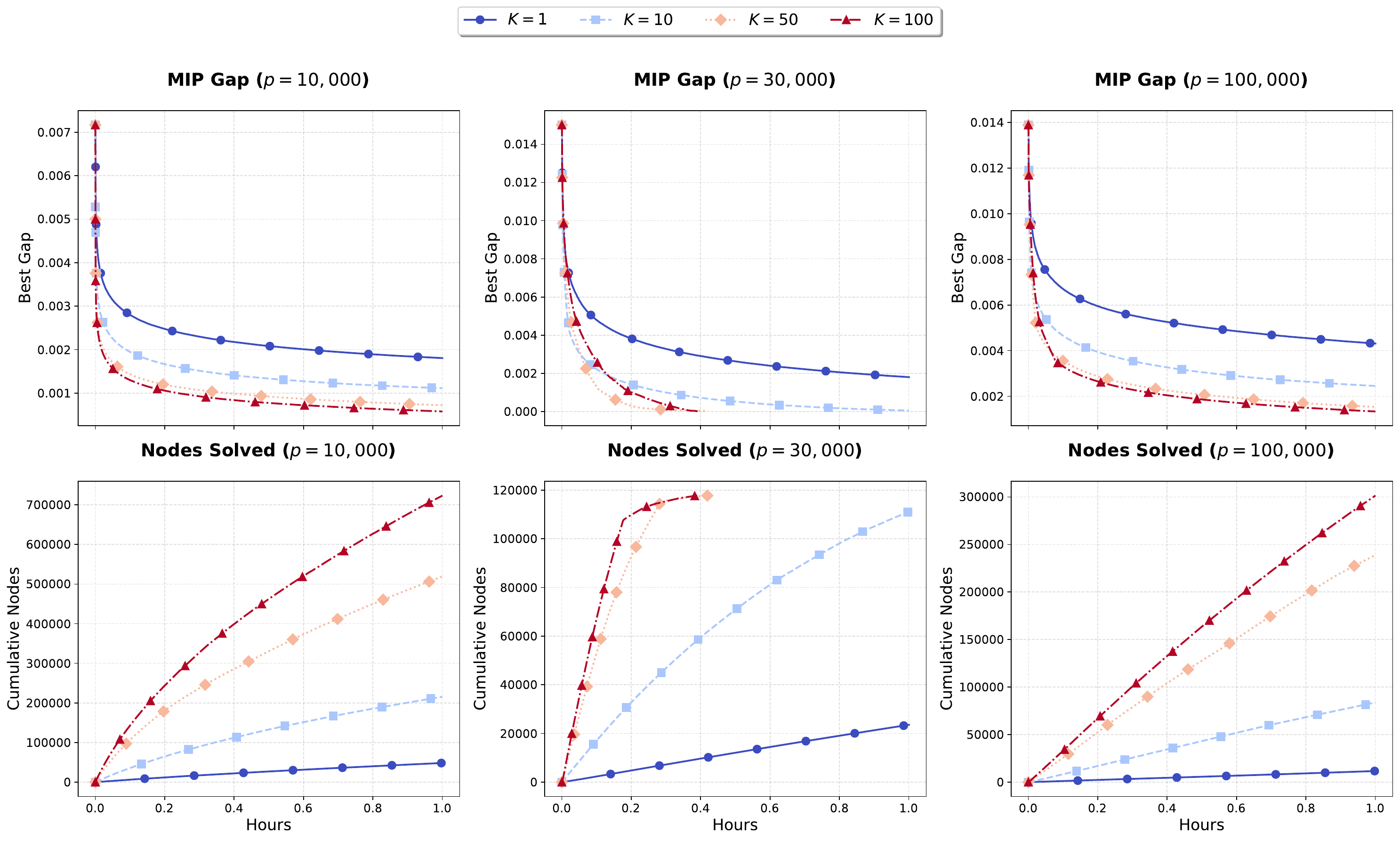}
    \caption{Synthetic examples: MIP gap (top) and number of nodes processed (bottom) with different batch sizes.}
    \label{fig:node_parallel_syn}
\end{figure}

\begin{table}[h]
\centering
\setlength{\tabcolsep}{4pt}
\renewcommand{\arraystretch}{1.05}
\begin{tabular}{l | cc | cc | cc}
\hline
 & \multicolumn{2}{c|}{$p = 10^4$} & \multicolumn{2}{c|}{$p = 3\times10^4$} & \multicolumn{2}{c}{$p = 10^5$} \\
Batch ($K$) & Mem (GB) & Rel. Inc. & Mem (GB) & Rel. Inc. & Mem (GB) & Rel. Inc. \\
\hline
1   & 0.636 & 1.00x & 1.285 & 1.00x & 2.824 & 1.00x \\
10  & 0.652 & 1.03x & 1.310 & 1.02x & 2.983 & 1.06x \\
50  & 0.726 & 1.14x & 1.459 & 1.14x & 3.693 & 1.31x \\
100 & 0.816 & 1.28x & 1.696 & 1.32x & 4.580 & 1.62x \\
\hline
\end{tabular}
\caption{Peak memory usage (GB) and relative increase in peak memory compared to the $K=1$ baseline as $p$ varies.}
\end{table}

\noindent \textbf{Real-World Dataset.}
We also evaluate node parallelism on the Census response dataset from~\cite{Ibrahim_2025}, a real-world regression problem in which the dependent variable is the self-response rate from the 2013--2017 American Community Survey. We augment the original feature set with all degree-two interaction terms, giving a dataset with $n = 43{,}197$ samples and $p = 43{,}955$ features. This is a large-scale problem in both $n$ and $p$, which causes challenges even for state-of-the-art algorithms such as {L0BnB}. Figure~\ref{fig:census_gap_nodes_relative} reports the MIP optimality gap over wall-clock time (top) and the cumulative number of nodes solved since the first node was solved (bottom), for batch sizes $K \in \{1, 5, 10, 30\}$ alongside {L0BnB}. The results clearly demonstrate that node parallelism substantially accelerates both metrics relative to sequential node processing ($K=1$). Larger batch sizes solve significantly more nodes within the same time budget, and correspondingly achieve much smaller MIP gaps over the course of the run. For instance, $K=30$ processes over 10 times as many nodes as $K=1$ within the allotted time and consistently maintains a lower optimality gap throughout.

\begin{figure}[H]
    \centering
    \includegraphics[width=0.6\linewidth]{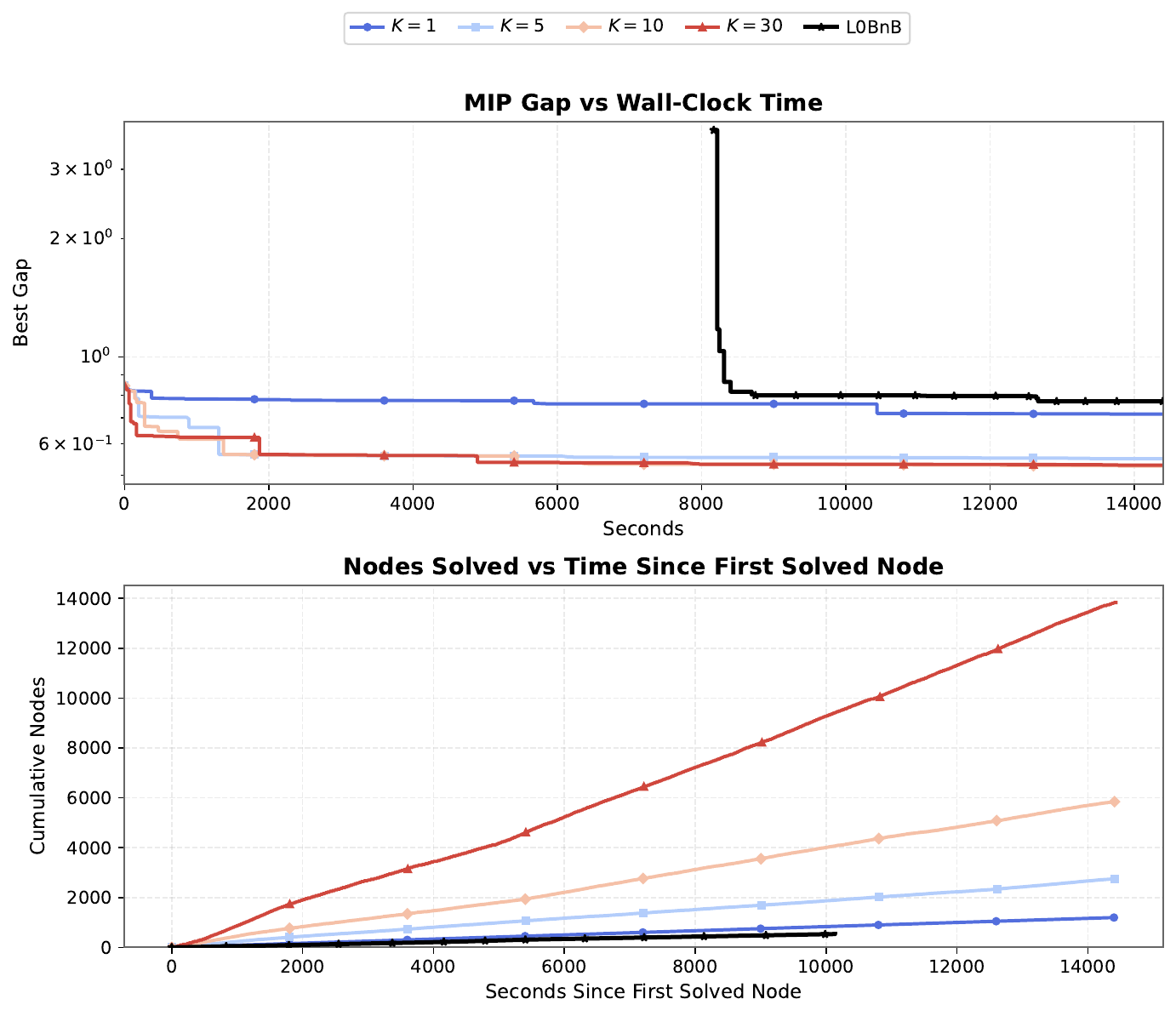}
    \caption{Census dataset: MIP optimality gap over wall-clock time (top) and cumulative nodes solved since the first solved node (bottom), for  {GPUBnB} with batch sizes $K \in \{1, 5, 10, 30\}$ and  {L0BnB}. Larger batch sizes substantially reduce the optimality gap and process far more nodes within the same time budget.  {L0BnB} struggles on this large-scale instance due to the high per-node cost imposed by large $n$ and $p$.}
    \label{fig:census_gap_nodes_relative}
\end{figure}

\section*{Acknowledgment}

The authors acknowledge research support funding from the Office of Naval Research Grant N000142512504.

\bibliographystyle{plain}
\bibliography{arXiv/references}

\appendix
\section{Additional Technical Details}\label{sect:app-technical}

\smallskip

\subsection{Derivation of problem \eqref{eq:relaxnode2}} We simplify problem \eqref{eq:relaxnode} by analytically solving for the optimal values of variables $z$ and $s$ for any fixed, feasible $\beta$, and then substituting these optimal solutions back into the original problem. This reduction allows us to reformulate the problem solely in terms of $\beta$.

Consider any feasible solution $(\beta, z, s)$ to problem \eqref{eq:relaxnode}, it follows $|\beta_i|\le Mz_i \le M,\,\forall i\in[p]$. The minimization with respect to $z$ and $s$ can be decomposed into $p$ independent subproblems, one for each coordinate $i \in [p]$. Specifically, for each $i \in [p]$, we solve:
\begin{equation} \label{eq:minzs}
\begin{aligned}
 \min _{z_i, s_i} \quad &\lambda_0 z_i +\lambda_2 s_i \\
\text { s.t.} \quad & \beta_i^2 \leq s_i z_i, \,\, |\beta_i| \leq M z_i, \\
& z_i \in [0,1],\, s_i \geq 0,  \\
& z_i=0 \text{ if }\,i \in \mathcal{F}_0,\, z_i=1\text{ if }i \in \mathcal{F}_1.
\end{aligned}
\end{equation}
We discuss the following scenarios.
\begin{enumerate}[ left=0pt, labelsep=1em]
    \item  $i \in \mathcal{F}_0$. When $|\beta_i|\neq 0$, problem \eqref{eq:minzs} is infeasible; when $\beta_i=0$, the only feasible solution is $(z_i^*,s_i^*)=(0,0)$, and the minimum is $0$.
    \item $i \in \mathcal{F}_1$. The optimal solution to problem \eqref{eq:minzs} is $(z_i^*,s_i^*)=(1,\beta_i^2)$, and the minimum is $\lambda_0+\lambda_2\beta_i^2$.
    \item $i \notin \mathcal{F}_0 \cup \mathcal{F}_1$. Note that $s_i=\beta_i^2/z_i$ is the smallest possible value $s$, which satisfies the constraint of problem \eqref{eq:minzs} (for the case $\beta_i=z_i=0$, define $\beta_i^2/z_i=0$). Problem \eqref{eq:minzs} then reduced to 
    \begin{equation} \label{eq:minz}
    \begin{aligned}
     \min _{z_i} \quad &\lambda_0 z_i +\lambda_2 \frac{\beta_i^2}{z_i} \\
    \text { s.t.} \quad &  \frac{|\beta_i|}{M} \leq z_i\le 1. \\
    \end{aligned}
    \end{equation}
    Note that the objective function $z_i \mapsto \lambda_0 z_i +\lambda_2 \frac{\beta_i^2}{z_i}$ is convex and its gradient reaches zero when $z_i=\sqrt{\frac{\lambda_2}{\lambda_0}}|\beta_i|$.  We consider three cases: 
    \begin{enumerate}
        \item$1\le \sqrt{\frac{\lambda_2}{\lambda_0}}|\beta_i|$. The objective is monotonically decreasing within the feasible region, therefore $(z_i^*,s_i^*)=(1,\beta_i^2/z_i^*)$ and the minimum is $\lambda_0+\lambda_2\beta_i^2$.
        \item $\frac{|\beta_i|}{M} \leq \sqrt{\frac{\lambda_2}{\lambda_0}}|\beta_i| \le 1$. By optimality condition, $(z_i^*,s_i^*)=\left(\sqrt{\frac{\lambda_2}{\lambda_0}}|\beta_i|,\beta_i^2/z_i^*\right)$ and the minimum is $2\sqrt{\lambda_0\lambda_2}|\beta_i|$.
        \item $\sqrt{\frac{\lambda_2}{\lambda_0}}|\beta_i|\le \frac{|\beta_i|}{M}$. The objective is monotonically increasing within the feasible region, therefore $(z_i^*,s_i^*)=\left(\frac{|\beta_i|}{M},\beta_i^2/z_i^*\right)$, the minimum is $\left(\frac{\lambda_0}{M}+\lambda_2M\right)|\beta_i|$.
    \end{enumerate}
\end{enumerate}
Summing up all the above cases, we arrive at the simplified problem \eqref{eq:relaxnode2}.

\smallskip

\textbf{Derivation of equation \eqref{eq:minbetalower}}. For any given $i\in [p]$ and $\tilde \beta_i$, we aim to derive
\begin{equation}\label{eq:minbetalower-app}
\beta_i^*=\argmin_{\beta_i} \frac{\rho}{2} (\beta_i - \tilde{\beta_i})^2 + \psi_i\left(\beta_i ; \lambda_0, \lambda_2, M\right) + \infty \cdot \mathbf{1}_{\left\{|\beta_i| \leq M\right\}} 
\end{equation}
We discuss the following scenarios.
\begin{enumerate}[ left=0pt, labelsep=1em]
    \item  $i \in \mathcal{F}_0$. In this case $\psi_i\left(\beta_i ; \lambda_0, \lambda_2, M\right)= \infty\cdot \textbf{1}_{\{\beta_i=0\}}$, and it is clear that $\beta_i^*=0$.
    \item $i \in \mathcal{F}_1$. In this case $\psi_i\left(\beta_i ; \lambda_0, \lambda_2, M\right)= \lambda_0 + \lambda_2 \beta_i^2$, and it follows from first-order optimality conditions that \begin{equation}
        \beta_i^*=\operatorname{proj}_{[-M, M]} \left( \frac{\rho}{\rho + 2 \lambda_2} \tilde{\beta}_i \right) = T\left(\frac{\rho}{\rho+2\lambda_2}\tilde{\beta_i},0,M\right).
    \end{equation}
    \item $i \notin \mathcal{F}_0 \cup \mathcal{F}_1$ and $\sqrt{\frac{\lambda_0}{\lambda_2}}\le M$. In this case 
    \begin{equation}
       \psi_i\left(\beta_i ; \lambda_0, \lambda_2,M\right)= \begin{cases}2 \sqrt{\lambda_0 \lambda_2}\left|\beta_i\right| & \text { if }\left|\beta_i\right| \leq \sqrt{\lambda_0 / \lambda_2} \\ \lambda_0+\lambda_2 \beta_i^2 & \text { else if }\sqrt{\lambda_0 / \lambda_2}\le \left|\beta_i\right| \le M.\end{cases}
    \end{equation}
     \begin{enumerate}
     \item If $|\beta_i^*| \le \sqrt{\lambda_0 / \lambda_2}$, it follows from the first-order optimality conditions that $\beta_i^*$ is given by the boxed version of the soft thresholding operator \cite{friedman2010regularization}, i.e., $\beta_i^*=T\left(\tilde{\beta}_i ; \frac{2\sqrt{\lambda_0 \lambda_2}}{\rho}, M\right)$. This is valid as long as $\left|T\left(\tilde{\beta}_i ; \frac{2\sqrt{\lambda_0 \lambda_2}}{\rho}, M\right)\right|\le \sqrt{\frac{\lambda_0}{\lambda_2}}$, which implies $|\tilde{\beta}_i| \leq \frac{2\sqrt{\lambda_0 \lambda_2}}{\rho}+\sqrt{\frac{\lambda_0}{\lambda_2}}$.
     \item If $|\beta_i^*| > \sqrt{\lambda_0 / \lambda_2}$, similar to Case 2, $\beta_i^*=T\left(\frac{\rho}{\rho+2\lambda_2}\tilde{\beta_i},0,M\right)$. This is valid as long as $|T\left(\frac{\rho}{\rho+2\lambda_2}\tilde{\beta_i},0,M\right)|> \sqrt{\frac{\lambda_0}{\lambda_2}}$, which implies $|\tilde{\beta}_i| > \frac{2\sqrt{\lambda_0 \lambda_2}}{\rho}+\sqrt{\frac{\lambda_0}{\lambda_2}}$.
     \end{enumerate}
    \item $i \notin \mathcal{F}_0 \cup \mathcal{F}_1$ and $\sqrt{\frac{\lambda_0}{\lambda_2}} > M$. In this case $\psi_i\left(\beta_i ; \lambda_0, \lambda_2, M\right)= \left(\lambda_0 / M+\lambda_2 M\right)\left|\beta_i\right|$. Similar to Case 3(b), the optimal solution is given by the boxed version of the soft thresholding operator $\beta_i^*=T\left(\tilde{\beta}_i ; \frac{\lambda_0}{M\rho}+\frac{\lambda_2 M}{\rho}, M\right)$.
\end{enumerate}
Summing up all the above cases, we arrive at equation \eqref{eq:minbetalower}.

\smallskip

\textbf{Proof of Proposition \ref{thm:dual}}. The Lagrangian function of problem \eqref{eq:ADMM1} is 
\begin{equation} \label{eq:ADMMlag}
L(b,\beta, v) =\frac{1}{2}\|y-X b\|_2^2+\sum_{i \in[p]} \psi_i\left(\beta_i ; \lambda_0, \lambda_2, M\right)  + \infty \cdot \mathbf{1}_{\left\{\|\beta\|_{\infty} \leq M\right\}} +v^\top(b-\beta)
\end{equation}
The Lagrangian dual reads:
\begin{equation} \label{eq:ADMMdual}
\max_{v\in\mathbb{R}^p} \min_{b,\beta\in\mathbb{R}^p} \frac{1}{2}\|y-X b\|_2^2+\sum_{i \in[p]} \psi_i\left(\beta_i ; \lambda_0, \lambda_2, M\right)  + \infty \cdot \mathbf{1}_{\left\{\|\beta\|_{\infty} \leq M\right\}} +v^\top(b-\beta)
\end{equation}
The inner minimization problem of \eqref{eq:ADMMdual} decomposes into the following separable minimization subproblems:
\begin{equation} \label{eq:ADMMdual2}
\begin{aligned}
 & \min_{b,\beta} \quad  \frac{1}{2}\|y-X b\|_2^2+\sum_{i \in[p]} \psi_i\left(\beta_i ; \lambda_0, \lambda_2, M\right)  + \infty \cdot \mathbf{1}_{\left\{\|\beta\|_{\infty} \leq M\right\}} +v^\top(b-\beta) \\
  = &\, \min_b \frac{1}{2}\|y-X b\|_2^2 + v^\top b + \sum_{i \in[p]} \left(\min_{\beta_i}\psi_i\left(\beta_i ; \lambda_0, \lambda_2, M\right) + \infty \cdot \mathbf{1}_{\left\{|\beta_i| \leq M\right\}}-v_i\beta_i\right) 
\end{aligned}
\end{equation}
If $v \notin \text{row}(X)$, there exists $b' \in \text{null}(X)$ with $v^\top b' \neq 0$. Then for any feasible solution $b$, the vector $b + \alpha b'$ remains feasible for all $\alpha \in \mathbb{R}$, and $v^\top(b + \alpha b')$ can be made arbitrarily small. This implies the minimization problem is unbounded below. Therefore, any dual feasible $v$ must lie in $\text{row}(X)$, i.e., $v = X^\top r$ for some $r \in \mathbb{R}^n$. Substituting this representation into \eqref{eq:ADMMdual2} yields
\begin{equation} \label{eq:ADMMdual3}
 \min_b \frac{1}{2}\|y-X b\|_2^2 +r^\top Xb + \sum_{i \in[p]} \left(\min_{\beta_i}\psi_i\left(\beta_i ; \lambda_0, \lambda_2, M\right) + \infty \cdot \mathbf{1}_{\left\{|\beta_i| \leq M\right\}}-(X_i^\top r)\beta_i\right) 
\end{equation}
It follows from the first-order optimality condition that the minimizer $b^*$ of the first subproblem satisfies $r=y-Xb^*$, and the minimum is $-\frac12\|r\|_2^2+y^\top r$. Next, we discuss how to solve the following problem for a given $i\in[p]$:
\begin{equation} \label{eq:ADMMdual4}
\min_{\beta_i}\psi_i\left(\beta_i ; \lambda_0, \lambda_2, M\right) + \infty \cdot \mathbf{1}_{\left\{|\beta_i| \leq M\right\}}-(X_i^\top r)\beta_i
\end{equation}
We discuss the following scenarios.
\begin{enumerate}[ left=0pt, labelsep=1em]
    \item  $i \in \mathcal{F}_0$. In this case $\psi_i\left(\beta_i ; \lambda_0, \lambda_2, M\right)= \infty\cdot \textbf{1}_{\{\beta_i=0\}}$, and it is clear that $\beta_i^*=0$ and the minimum is $0$.
    \item $i \in \mathcal{F}_1$. In this case $\psi_i\left(\beta_i ; \lambda_0, \lambda_2, M\right)= \lambda_0 + \lambda_2 \beta_i^2$, and it follows from first-order optimality conditions that $\beta_i^*=\operatorname{proj}_{[-M, M]} \left( \frac{X_i^\top r}{ 2 \lambda_2}  \right)$. The minimum of \eqref{eq:ADMMdual4}, obtained by substituting $\beta_i^*$, is $-h(|X_i^\top r|)$. Here, $h:\mathbb{R}_{\ge0}\to\mathbb{R}$ is defined as 
    \begin{equation}\label{eq:hdef2}
    h(x)  =  \begin{cases} x^2/4\lambda_2-\lambda_0& \text { if } x\le 2M\lambda_2\\
    Mx-\lambda_0-\lambda_2M^2 & \text { otherwise.} 
    \end{cases}
    \end{equation}
    \item $i \notin \mathcal{F}_0 \cup \mathcal{F}_1$ and $\sqrt{\frac{\lambda_0}{\lambda_2}}\le M$. In this case 
    \begin{equation}
       \psi_i\left(\beta_i ; \lambda_0, \lambda_2,M\right)= \begin{cases}2 \sqrt{\lambda_0 \lambda_2}\left|\beta_i\right| & \text { if }\left|\beta_i\right| \leq \sqrt{\lambda_0 / \lambda_2} \\ \lambda_0+\lambda_2 \beta_i^2 & \text { else if }\sqrt{\lambda_0 / \lambda_2}\le \left|\beta_i\right| \le M.\end{cases}
    \end{equation}
    For the minimizer $\beta_i^*$: if $\sqrt{\lambda_0 / \lambda_2} \leq |\beta_i^*| \leq M$, then the minimum is $-h(|X_i^\top r|)$ following Case 2; if $|\beta_i^*| < \sqrt{\lambda_0 / \lambda_2}$, then by first-order optimality conditions, $\beta_i^* = 0$ and the minimum is 0. Moreover, since $\sqrt{\lambda_0 \lambda_2}|\beta_i| \leq \lambda_0 + \lambda_2\beta_i^2$, the minimum from Case 2 cannot exceed that from Case 3. This implies $0 \leq -h(|X_i^\top r|)$ when $|\beta_i^*| < \sqrt{\lambda_0 / \lambda_2}$. Thus, the minimum can be compactly expressed as $-\left[h(|X_i^\top r|)\right]_+$.
    \item $i \notin \mathcal{F}_0 \cup \mathcal{F}_1$ and $\sqrt{\frac{\lambda_0}{\lambda_2}} > M$. In this case $\psi_i\left(\beta_i ; \lambda_0, \lambda_2, M\right)= \left(\lambda_0 / M+\lambda_2 M\right)\left|\beta_i\right|$. By the first-order optimality conditions, $\beta_i^*$ can only take three values: $0$, $-M$, or $M$. Comparing these cases, the minimum value is $-\left[M|X_i^\top r|-\lambda_0-\lambda_2M^2\right]_+$.
\end{enumerate}
Summing up all the above cases, the Lagrangian dual can be expressed as \begin{equation}\label{eq:dual2}
\max_{r\in \mathbb{R}^n} \,\,\,-\frac{1}{2}\|r\|_2^2+y^\top r -\sum_{i\in[p]} \nu_i\left( \left|X_i^\top r\right|\right) 
\end{equation}
where 
\begin{equation}\label{eq:nudef2}
\nu_i(x)   = \begin{cases}
0, & \text{if } i \in \mathcal{F}_0,\\
h(x), & \text{if } i \in \mathcal{F}_1,\\
[h(x)]_+, & \text{if } i \notin \mathcal{F}_0 \cup \mathcal{F}_1 \text{ and } \sqrt{\frac{\lambda_0}{\lambda_2}} \le M,\\
\left[ Mx - \lambda_0 - \lambda_2 M^2 \right]_+, & \text{if } i \notin \mathcal{F}_0 \cup \mathcal{F}_1 \text{ and } \sqrt{\frac{\lambda_0}{\lambda_2}} > M.
\end{cases}
\end{equation}
Strong duality holds for \eqref{eq:ADMM1} since it satisfies Slater’s condition \cite{bertsekas1997nonlinear}.

\subsection{Matching Pursuit Algorithm}\label{sec:matching_pursuit}

As outlined in the main text, we implement a matching pursuit-inspired algorithm at the root node, with forward and backward phases. We now explain how the algorithm works in detail. During the forward phase, we consider adding variables \( j \notin S \) to the current support set \( S \). For each candidate, we solve the optimization problem:

\[
\beta_j = \arg\min_{\beta_j \in [-M, M]} \left\{ \tfrac{1}{2} \| r - X_j \beta_j \|_2^2 + \lambda_2 \beta_j^2 + \lambda_0 \right\},
\]

where \( r = y - X_S \beta_S \) is the residual vector. The solution involves computing the unconstrained minimizer \( \beta_j^\ast = \frac{c_j}{D_j} \), where $c_j = X_j^\top r, \quad D_j = \| X_j \|_2^2 + 2\lambda_2$, followed by projection onto \([-M, M]\):
\[
\beta_j = \operatorname{Proj}_{[-M, M]} \left( \beta_j^\ast \right) = \min\{ \max\{ \beta_j^\ast, -M \}, M \}.
\]
The objective change for candidate $j$ keeping other coordinates frozen is given by:
\[
\Delta_j = -\beta_j c_j + \tfrac{1}{2} D_j \beta_j^2 + \lambda_0.
\]
Similarly, for backward selection, we fix the residuals and assess the impact of removing variables \( j \in S \) by setting \( \beta_j = 0 \):
\[
\Delta_j = \beta_j c_j + \left( \tfrac{1}{2} \| X_j \|_2^2 - \lambda_2 \right) \beta_j^2 - \lambda_0.
\]
For both forward and backward steps, these values can be batched and calculated for all candidate variables using vectorized operations.



\end{document}